\documentclass{siamonline0316}

\usepackage{amsfonts}
\usepackage{graphicx}
\usepackage{epstopdf}
\usepackage{algorithmic}
\usepackage{array,colortbl}
\usepackage{amsmath,amsfonts,amssymb,bm} 

\usepackage{verbatim}

\ifpdf
  \DeclareGraphicsExtensions{.eps,.pdf,.png,.jpg}
\else
  \DeclareGraphicsExtensions{.eps}
\fi

\newcommand{\TheTitle}{Adaptive Multilevel Monte Carlo Approximation of Distribution Functions} 
\newcommand{\TheAuthors}{M. B. Giles, T. Nagapetyan, and K. Ritter}

\headers{\TheTitle}{\TheAuthors}

\title{\TheTitle}

\author{Mike B. Giles\thanks{Mathematical Institute, University of Oxford,
Andrew Wiles Building, Woodstock Rd, Oxford OX2 6GG,
England, \email{mike.giles@maths.ox.ac.uk}} 
\and Tigran Nagapetyan\thanks{Department of Statistics, University of 
Oxford, 24--29 St Giles', Oxford OX1 3LB, England, 
\email{nagapet\-yan@stats.ox.ac.uk}} \and 
Klaus Ritter\thanks{Fachbereich Mathematik, 
TU Kaiserslautern, 67653 Kaiserslautern, Germany,
\email{ritter@mathematik.uni-kl.de}}}

\usepackage{amsopn}

\newcommand{\R}{\mathbb{R}} 
\newcommand{\N}{\mathbb{N}} 
\newcommand{\eps}{\varepsilon}
\newcommand{\Eps}{\epsilon}
\newcommand{\M}{{\mathcal M}}
\newcommand{\SSS}{{\mathcal S}}
\newcommand{\MM} {\M^{k,\delta,L_0,L_1}_{N_{L_0}, \dots, N_{L_1}}}

\newcommand{\A}{{\mathcal A}}
\newcommand{\tl}{Y^{(\ell)}}
\newcommand{\Exp}    {\operatorname{E}}
\newcommand{\V}    {\operatorname{Var}}
\newcommand{\err}  {\operatorname{error}}
\newcommand{\cost} {\operatorname{cost}}

\newcommand{\bt}{\mathbf{t}}
\newcommand{\ba}{\mathbf{a}}
\newcommand{\ls}{{\scriptstyle \;\lesssim\;}}

\newtheorem{remark}{Remark}
\newtheorem{exmp}{Example}

\ifpdf
\hypersetup{
  pdftitle={\TheTitle},
  pdfauthor={\TheAuthors}
}
\fi

\begin{document}

\maketitle

\begin{abstract}
We analyse a multilevel Monte Carlo method for the approximation 
of distribution functions of univariate random variables.
Since, by assumption, the target distribution is not known
explicitly, approximations have to be used. 
We provide an asymptotic analysis of the error and the cost 
of the algorithm. 
Furthermore we construct an adaptive version of the algorithm
that does not require any a priori knowledge on weak or strong
convergence rates. We apply the adaptive algorithm
to smooth path-independent and path-dependent functionals and 
to stopped exit times of SDEs.
\end{abstract}



\section{Introduction}
Let $Y$ denote a real-valued random variable
with distribution function $F$.
We study the approximation of $F$ with 
respect to the supremum norm on a compact interval
$[S_0,S_1]$, without assuming that
the distribution of $Y$ is explicitly known
or that the simulation of $Y$ is feasible. Instead,
we suppose that a sequence of random variables $\tl$ is
at hand that converge to $Y$ in a suitable way and
that are suited to simulation.

The general approach for this problem, based on the multilevel 
Monte Carlo (MLMC) approach \cite{heinrich98,giles08}, 
has been presented in \cite{giles2014multi} and applied 
in the context of stochastic differential equations (SDEs). 
The suggested algorithm has already been used in stochastic oil 
reservoir simulations, which is based on  numerical approximation of 
large-scale stochastic subsurface simulations \cite{WRCR:WRCR22402}, 
while the suggested smoothing technique has been used for analysis 
in \cite{plechavc2016multi}.

In outline, one approach is to use the standard MLMC algorithm
to approximate $F(s_i) = \Exp( 1_{]-\infty, s_i]} (Y) )$, for a finite
set of $k$ spline points $s_i$, and then use spline interpolation
to define the approximation to $F(s)$ for the whole interval.
The drawback of this approach is that the discontinuity in the
indicator function leads to a high variance for the MLMC estimator.
Instead, in \cite{giles2014multi} we introduce smoothing and 
approximate $\Exp( g((Y-s_i)/\delta) )$, where $g$ is a smooth 
approximation to the indicator function $1_{\left]-\infty,0\right]}$, 
and $\delta \ll S_1-S_0$.  This leads
to an approximation of $F$ which has four error components:
\begin{itemize}
\item
spline interpolation error, which depends on the number $k$ of
spline knots;
\item
smoothing error, which depends on the smoothing
parameter $\delta$;
\item
discretisation bias, which additionally
depends on the accuracy of $Y^{(L)}$ on the finest level $L$;
\item
Monte Carlo sampling error, which additionally depends on the number 
$N_\ell$ of samples on each level.
\end{itemize}

The standard MLMC algorithm has analysis and heuristics to
determine near-optimal values for $L$ and $N_\ell$
\cite{giles08,giles15}.  This paper
addresses the question of how to adaptively determine the
values for $k$ and $\delta$ to minimise the computational
cost to achieve a prescribed accuracy.  This extends the
asymptotic analysis in \cite{giles2014multi} 
which, roughly speaking,
assumes that the orders of weak and strong
convergence of $Y^{(\ell)}$ towards $Y$ are known a priori.

The paper is organised as follows. In Section \ref{s2} we recall briefly 
the algorithm strategy, and present updated bounds on the cost and the 
assumptions under which they were derived. In Section \ref{err:dec} we 
present the error decomposition.
Section \ref{s3} analyses the complexity of the MLMC
algorithm, i.e., 
it provides a new asymptotic upper bound of the cost of the 
MLMC algorithm in terms of its error.
In Section \ref{sec:algorithm} we describe the adaptive 
approach for distribution function approximation. 
Finally, Section \ref{s6} is devoted to the numerical experiments. 

For an alternative approach to the use of MLMC to construct approximations
of distribution functions with error analysis in different norm, see the recent work of Bierig and Chernov 
\cite{bc15,bc16} who use the Maximum Entropy method to approximate 
the distribution, and MLMC to obtain estimates for the required moments.

\section{Approximation of Distribution Functions on Compact Intervals
via MLMC}\label{s2}

In this section we present the multilevel algorithm for the
approximation of a distribution function $F$ of
a real-valued random variable $Y$ on a compact interval
$[S_0,S_1]$. In particular, we discuss the smoothing
and interpolation steps of the algorithm, which are
not present in the standard setting for MLMC, 
namely the approximation of the expectation $\Exp(Y)$.
Our approach basically follows \cite{giles2014multi}.
However, through an improvement in the implementation of
the algorithm we derive a new cost bound, which leads to an
improved upper bound for the cost of the algorithm in terms of 
its error, see Section \ref{r3}.

\subsection{Smoothing}\label{s2.1}

For the approximation of the distribution function
$F$ of $Y$ a straight-forward application of the 
MLMC approach based on
\[
F(s) = 
\Exp(1_{\left]-\infty,s\right]} (Y))
\]
would suffer from the discontinuity of
$1_{\left]-\infty,s\right]}$. This can be avoided 
by a smoothing step, provided that a density exists and is
sufficiently smooth. 
Specifically, we assume that
\begin{enumerate}
\item[(A1)]
the random variable $Y$ has a density $\rho$ on $\R$
that is $r$-times continuously differentiable on
$[S_0-\delta_0,S_1+\delta_0]$
for some $r \in \N_0$ and $\delta_0>0$.
\end{enumerate}

The smoothing is based on rescaled translates of a 
function $g : \R \to \R$ with 
the following properties:
\begin{enumerate}
\item[(S1)]
The cost of computing $g(s)$ is bounded by a constant, uniformly 
in $s \in \R$.
\item[(S2)]
$g$ is Lipschitz continuous.
\item[(S3)]
$g(s) = 1$ for $s<-1$ and $g(s)=0$ for $s>1$.
\item[(S4)]
$\int_{-1}^1 s^j \cdot 
(1_{\left]-\infty,0\right]} (s) - g(s)) \, ds = 0$ for
$j=0,\dots,r-1$.
\end{enumerate}
Obviously, $g$ is bounded due to (S2) and (S3).

\begin{remark}\label{r2}
Such a function $g$ is easily constructed as follows. 
There exists a uniquely determined
polynomial $p$ of degree at most $r+1$ such that
\[
\phantom{\qquad\quad j=0,\dots,r-1,}
\int_{-1}^1 s^j \cdot p(s) \, ds = (-1)^{j}/(j+1),
\qquad\quad j=0,\dots,r-1,
\]
as well as $p(1) = 0$ and $p(-1) = 1$.
The extension $g$ of $p$ with $g(s) = 1$ for $s < -1$ and
$g(s) = 0$ for $s> 1$ has the properties as claimed.
Since $g-1/2$ is an odd function, the same function $g$
arises in this way for $r$ and $r+1$, if $r$ is even.

For $r=3$, which will be considered in the numerical experiments,
we obtain
\[
p(s) = \frac{1}{2} + \frac{5s^3 - 9s}{8}. 
\]
\end{remark}

Using $\|\cdot\|_\infty$ to denote the supremum norm on $C([S_0,S_1])$,
we have the following estimate for the bias that is induced
by smoothing with parameter $\delta$, i.e., by approximation of
$1_{\left]-\infty,s\right]}$ by $g((\cdot -s)/\delta)$,
see \cite[Lemma~2.2]{giles2014multi}.

\begin{lemma}\label{l1}
There exists a constant $c>0$ such that
\[
\|F - \Exp(g((Y -\cdot )/\delta)) \|_\infty \leq c \cdot \delta^{r+1}
\]
holds for all $\delta \in \left]0,\delta_0\right]$.
\end{lemma}

\begin{proof}
Clearly
\begin{align*}
F(s) - \Exp (g((Y-s)/\delta) )
& = 
\int_{-\infty}^\infty \rho(u) \cdot 
(1_{\left]-\infty,s\right]} (u) - g((u-s)/\delta)) \, du\\
& =
\delta \cdot \int_{-1}^{1} \rho(s+\delta u) \cdot 
(1_{\left]-\infty,0\right]} (u) - g(u)) \, du,
\end{align*}
so that the statement follows in the case $r=0$.
For $r \geq 1$ the Taylor expansion
\[
\rho(s+\delta u) = \sum_{j=0}^{r-1}
\rho^{(j)}(s) \cdot (\delta u)^j / j! 
+ R(\delta u, s)
\]
yields
\begin{align*}
&F(s) - \Exp (g((Y-s)/\delta)) = \delta \cdot
\int_{-1}^{1} R(\delta u, s) \cdot
(1_{\left]-\infty,0\right]} (u) - g(u)) \, du\\
&\qquad = \frac{\delta^{r+1}}{(r-1)!} \int_{-1}^1\! \int_0^u
(u-t)^{r-1} \cdot \rho^{(r)} (s+\delta t)\, dt \cdot
(1_{\left]-\infty,0\right]} (u) - g(u)) \, du,
\end{align*}
which completes the proof.
\end{proof}

\begin{remark}\label{r1}
Consider $g$ according to Remark \ref{r2}, and assume that
$r$ is odd. Since 
\begin{align*}
&\int_{-1}^1\! \int_0^u
(u-t)^{r-1} \cdot \rho^{(r)} (s+\delta t)\, dt \cdot
(1_{\left]-\infty,0\right]} (u) - g(u)) \, du \\
&\qquad =
\mbox{}- \int_{0}^1\! \int_t^1
(u-t)^{r-1} \cdot g(u) \, du \cdot \bigl(\rho^{(r)} (s+\delta t)+
\rho^{(r)}(s-\delta t)
\bigr) \, dt,
\end{align*}
we have
\begin{align*}
&|F(s) - \Exp (g((Y-s)/\delta))|\\
&\qquad  \leq
\frac{2 \delta^{r+1}}{(r-1)!} 
\int_{0}^1 \left| \int_t^1
(u-t)^{r-1} \cdot g(u) \, du \right| \, dt \cdot 
\sup_{s \in [S_0-\delta,S_1+\delta]} |\rho^{(r)}(s)|.
\end{align*}
In particular, for $r=3$,
\[
\int_t^1
(u-t)^{2} \cdot g(u) \, du = \frac{-1}{96} \cdot (t-1)^4 \cdot
(t^2+4t+1),
\]
and therefore 
\[
\|F - \Exp (g((Y-\cdot)/\delta))\|_\infty
\leq \frac{\delta^4}{280} \cdot
\sup_{s \in [S_0-\delta,S_1+\delta]} |\rho^{(3)}(s)|.
\]
\end{remark}

\subsection{Interpolation and Monotonicity Corrections}\label{s44}

In the sequel 
$|\cdot|_\infty$ denotes the $\ell_\infty$-norm on $\R^k$. 

The approximation of $F$ on the interval $[S_0,S_1]$ is based on 
its approximation at finitely many points
\[
S_0 \leq s_1 < \dots < s_k \leq S_1,
\]
followed by a suitable extension to $[S_0,S_1]$.
For the extension we take a sequence of mappings
$Q_k^{r} : \R^k \to C([S_0,S_1])$ with the following
properties for some constant $c>0$: 
\begin{enumerate}
\item[(E1)]
For all $k \in \N$ and $x \in \R^k$
the cost for computing $Q_k^{r}(x)$ is bounded by $c \cdot k$. 
\item[(E2)]
For all $k \in \N$ and 
$x,y \in \R^k$
\[
\|Q_k^{r}(x) - Q_k^{r}(y)\|_\infty \leq c \cdot |x-y|_\infty.
\]
\item[(E3)]
For all $k \in \N$
\[
\|F - Q_k^{r}(F(s_1), \dots, F(s_k))\|_\infty \leq c \cdot k^{-(r+1)}.
\]
\end{enumerate}
These properties are easily achieved, e.g., by
piecewise polynomial interpolation with degree $\max(r,1)$
at equidistant points.

In terms of the spacing parameter $\tau$ of equidistant points,
the interpolation error is of the order $\tau^{r+1}$.
Since the smoothing error is of the order $\delta^{r+1}$ then
to balance these errors one has to take $\tau$ and $\delta$ of the
same order. 

\begin{remark}\label{r3}
In the numerical experiments we will first consider
the following simple cubic interpolant which is linear 
with respect to its inputs.
Let $r=3$, assume that $k= 3n+1$ with $n \in \N$, and put 
\[
\tau = (S_1-S_0)/n
\]
as well as
\[
\phantom{\qquad\quad j=1,\dots,k.}
s_j = S_0+ (j-1) \cdot (S_1-S_0)/(k-1),
\qquad\quad j=1,\dots,k.
\]
Furthermore, let $Q_k^3$ denote the
piecewise polynomial interpolation
of degree three at four consecutive knots.
The classical error estimate for polynomial interpolation yields
\begin{equation}\label{g45}
\|F - Q_k^{3}(F(s_1), \dots, F(s_k))\|_\infty \leq 
\frac{\tau^4}{1944} \cdot \|\rho^{(3)}\|_\infty.
\end{equation}
Furthermore, the corresponding 
Lipschitz constant is easily computed
explicitly, giving
\begin{equation}\label{g44}
\|Q_k^{3}(x)-Q_k^{3}(y)\|_\infty \leq c \cdot |x-y|_\infty,
~~~~~~
c = \frac{7 \cdot(2\sqrt{7} + 1)}{27} \approx 1.63.
\end{equation}
%
\end{remark}

The resulting vectors of data $y=(y_1,\dots,y_k) \in \R^k$ generated by 
the multilevel construction are not necessarily non-decreasing. Even
when they are, there is the possibility that a piecewise polynomial
interpolant may not be non-decreasing.
Therefore we employ a two stage  post-processing for the resulting 
approximation as described in Remark \ref{r3a}.

\begin{remark}\label{r3a}
For $y = (y_1,\dots,y_k) \in \R^k$ we define $u_0=0$ and
\[
u_j = \min(\max(y_j,u_{j-1}),1)
\]
to get a non-decreasing sequence $u = (u_1,\dots,u_k) \in [0,1]$.
For symmetry reasons we also define
$v = (v_1,\dots,v_k) \in [0,1]$
with
\[
v_j = \max(\min(y_j,v_{j+1}),0),
\]
where $v_{k+1}=1$. 
Observe that 
$y \mapsto (u+v)/2$ defines a Lipschitz continuous mapping
with Lipschitz constant 1 w.r.t.\ $|\cdot|_\infty$.
In the second stage of the post-processing 
the piecewise polynomial interpolant $\varphi$ of $(u+v)/2$, 
as described in Remark \ref{r3}, is transformed into a non-decreasing
function that coincides with $\varphi$ at the knots $s_j$.
To this end we put
\[
f(s) = \min \Bigl(\, \max_{t \in [s_{j-1},s]} \varphi
(t),\varphi(s_j)\Bigr)
\]
and
\[
h(s) = \max \Bigl(\, \min_{t \in [s,s_{j}]}
\varphi(t),\varphi(s_{j-1})\Bigr)
\]
for $s \in [s_{j-1},s_j]$.
Observe that 
$\varphi \mapsto (f+h)/2$ defines a Lipschitz continuous mapping
with Lipschitz constant 1 w.r.t.\ $\|\cdot\|_\infty$.

Instead of the plain interpolation according to Remark \ref{r3},
in the numerical experiments we now consider
\[
Q_k^3 (y) = (f+h)/2.
\]
Clearly we have (E2) with $c$ given by \eqref{g44}.
Furthermore, $u=v=y$ for $y=(F(s_1),\dots,F(s_k))$
and 
\[
\max(\|F-f\|_\infty,\|F-h\|_\infty) \leq \|F - \varphi\|_\infty
\]
for every distribution function $F$ and every $\varphi \in
C([S_0,S_1])$. Consequently, \eqref{g45} holds true, 
and in particular we
have (E3). To actually compute $(f+h)/2$, which is a piecewise
polynomial of degree three, one has to solve at most $2k$
polynomial equations of degree two or three, and therefore
we also have (E1).
\end{remark}

\subsection{The MLMC Algorithm}\label{s2.3}

Our multilevel Monte Carlo construction is based on
a sequence $(Y^{(\ell)})_{\ell\in\N_0}$ of  
random variables, defined on a common probability space together
with $Y$.
Assumptions on the cost for the simulation of the joint
distribution of $Y^{(\ell)}$ and $Y^{(\ell-1)}$ and on the
weak and strong convergence will be specified
in Sections \ref{s2.4} and \ref{s2.2}.

For notational convenience we put
\[
\phantom{\qquad\quad t \in \R,}
g^{k,\delta} (t) = \left( g((t-s_1)/\delta), \dots,
g((t-s_k)/\delta) \right)
\in \R^k,
\qquad\quad t \in \R.
\]

We choose $L_0, L_1 \in \N_0$ with $L_0 \leq L_1$ as the
minimal and the maximal level, respectively,
and we choose replication
numbers $N_\ell \in \N$ for all levels $\ell=L_0,\dots,L_1$, as well as 
$k \in \N$
and $\delta \in \left]0,\delta_0\right]$. 
The corresponding MLMC algorithm for the approximation at the 
points $s_j$ is defined by
\[
\MM
=
\frac{1}{N_{L_0}} \cdot \sum_{i=1}^{N_{L_0}} 
g^{k,\delta} (Y^{(L_0)}_i) +
\sum_{\ell=L_0+1}^{L_1}
\frac{1}{N_\ell} \cdot \sum_{i=1}^{N_\ell} 
\left( g^{k,\delta} (Y^{(\ell)}_i) - g^{k,\delta} (Z^{(\ell)}_i)
\right) 
\]
with an independent family of $\R^2$-valued random variables
$(Y^{(\ell)}_i,Z^{(\ell)}_i)$ for 
$\ell=L_0,\dots,L_1$
and 
$i=1,\dots,N_\ell$ such that equality in distribution holds for
$(Y^{(\ell)}_i,Z^{(\ell)}_i)$ and $(Y^{(\ell)},Y^{(\ell-1)})$ 
if $\ell > L_0$ as well as for $Y^{(L_0)}_i$ and $Y^{(L_0)}$.

In the particular case $L=L_0=L_1$, i.e., in the single-level case,
we actually have a classical Monte Carlo algorithm,
based on independent copies of $Y^{(L)}$ only.
In addition to 
\[
\SSS^{k,\delta,L}_{N} =
\M^{k,\delta,L,L}_{N} =
\frac{1}{N} \cdot \sum_{i=1}^{N} g^{k,\delta} (Y^{(L)}_i)
\]
with $\delta >0$, we also consider the single-level algorithm
without smoothing. Hence we put
\[
\phantom{\qquad\quad t \in \R,}
g^{k,0} (t) = \left( 1_{\left]-\infty,s_1\right]}(t), \dots,
1_{\left]-\infty,s_k\right]}(t) \right)
\in \R^k,
\qquad\quad t \in \R,
\]
to obtain 
\[
\SSS^{k,0,L}_{N} = 
\M^{k,0,L,L}_{N} =
\frac{1}{N} \cdot \sum_{i=1}^{N} g^{k,0} (Y^{(L)}_i).
\]
Observe that $\SSS^{k,0,L}_{N}$ yields the values of the empirical 
distribution function, based on $N$ independent copies of $Y^{(L)}$,
at the points $s_j$. 

We employ $Q_k^{r}(\M)$ with 
\[
\M = \MM
\]
as a randomized algorithm for
the approximation of $F$ on $[S_0,S_1]$. 

\subsection{Cost Bounds}\label{s2.4}

In our analysis of the computational cost we assume 
the following for some constant $c>0$:
\begin{enumerate}
\item[(A2)] There exists a constant $M>1$ such that
the simulation of the joint
distribution of $Y^{(\ell)}$ and $Y^{(\ell-1)}$ is possible at cost
at most $c \cdot M^\ell$ for every $\ell \in \N$. 
\end{enumerate}
Typically, $M$ is a refinement factor, e.g., for the time-step
of a numerical approximation scheme for a stochastic differential
equation.

Together with the property (S1) for $g$, the assumption (A2) yields 
the following upper bound. There exists a constant $c>0$ such
that, for all $k,\delta,L_0,L_1,N_{L_0},\dots,N_{L_1}$,
\[
\cost\left(Q_k^{r}\left(\M\right)\right)
\leq c \cdot
\sum_{\ell=L_0}^{L_1} N_{\ell}\cdot (M^\ell+k)
\]
for the cost of $Q_k^r(\M)$, see
\cite[Eqn. (2.16)]{giles2014multi}.
In fact, for every replication on level $\ell$ the simulation
cost is at most of the order $M^\ell$, while the cost to evaluate
$g^{k,\delta}$ is of the order $k$. The number of arithmetic
operations that are needed additionally is of the order
$\sum_{\ell=L_0}^{L_1} N_\ell$, 
plus an additional order $k$ for the interpolation.

As an extension to the analysis in \cite{giles2014multi},
we derive an improved cost bound in the case of
equidistant knots. Observe that every realization of $\M$ is of the form
$\sum_{i=1}^N a_i g^{k,\delta}(t_i)$, where 
\begin{align*}
N &= N_{L_0} + 2 \sum_{\ell=L_0+1}^{L_1} N_\ell,\\
\ba& =(a_1,\dots,a_N) \in  \R^N
\end{align*}
with $|a_i| \in \{1/N_{L_0},\dots,1/N_{L_1}\}$,
and 
\[
\bt=(t_1,\dots,t_N) \in  \R^N
\]
being a realization
of $(Y_1^{(L_0)},\dots,Z_{N_{L_1}}^{(L_1)},Y_{N_{L_1}}^{(L_1)})$.
Taking into account that 
$g=1$
on $\left]-\infty,-1\right]$ and $g=0$ on $\left[1,\infty\right[$, 
the sum $\sum_{i=1}^N a_i g^{k,\delta}(t_i)$
may be evaluated in the following, efficient way, 
if the knots $s_j$ are chosen equidistantly. Put 
\[
j^*(t) = \min \{j \in \{1,\dots,k\} : s_j > t + \delta\}
\]
for $t \in \R$ with $\min \emptyset = \infty$, and
define $g_1^{k,\delta}(t) \in \R^k$ 
by 
\[
g_{1,j}^{k,\delta}(t) =
\begin{cases}
1, &\text{if $j=j^*(t)$,}\\
0, &\text{otherwise,}
\end{cases}
\]
for $j=1,\dots,k$.
Due to the uniform spacing of the knots,
the cost to compute $j^*(t)$ is bounded by
a constant, uniformly in $t$ and $k$. Hence
we have a uniform cost bound
of order $N$ to compute
\[
g_2^{k,\delta} (\ba,\bt) = \sum_{i=1}^N 
a_i g_1^{k,\delta}(t_i) \in \R^k
\]
for $N,k \in \N$ and $\ba,\bt \in  \R^N$.
Consider $\psi:\R^k \to \R^k$ given by
$\psi(x) = (x_1,x_1+x_2,\dots,\sum_{j=1}^k x_j)$,
which is sometimes called the scan function or the
cumulative summation function.
The vector
\[
g_3^{k,\delta}(\ba,\bt)= \psi(g_2^{k,\delta}(\ba,\bt)) \in \R^k
\]
of successive partial sums of $g_2^{k,\delta}(\ba,\bt)$
may therefore be computed at a cost uniformly bounded by a multiple
of $\max(N,k)$. Finally, let $g_4^{k,\delta}(t) \in \R^k$
be given by
\[
g_{4,j}^{k,\delta}(t) = 1_{[t-\delta,t+\delta]}(s_j)
\cdot g((t-s_j)/\delta)
\]
for $j = 1,\dots,k$ and $t \in \R$. Due to the uniform 
spacing of the knots, $g_4^{k,\delta}(t)$ may be
computed at a cost uniformly bounded by $\max(k \cdot \delta,1)$ 
for $t \in
\R$, $k \in \N$, and $\delta >0$. 
Since 
\[
\sum_{i=1}^N a_i g^{k,\delta} (t_i) = g_3^{k,\delta}(\ba,\bt)
+ \sum_{i=1}^N a_i g_4^{k,\delta}(t_i),
\]
the cost to compute this sum is uniformly bounded by
$k + N \cdot \max(k \cdot \delta, 1)$, up to a constant.
Altogether this yields the cost bound
\begin{equation}
\cost\left(Q_k^{r}\left(\M\right)\right)
\leq c \cdot
\left(
k + 
\sum_{\ell=L_0}^{L_1} N_{\ell}\cdot (M^\ell+k\cdot\delta)
\right)
\label{eq:cost:mlmc:new}
\end{equation}
with a constant $c>0$ for all
$k,\delta,L_0,L_1,N_{L_0},\dots,N_{L_1}$.
Typically $N_{L_0} \cdot M^{L_0}$ dominates $k$
and $M^{L_0}$ dominates $k \cdot \delta$,
so that the upper bound is of the order
$\sum_{\ell=L_0}^{L_1} N_{\ell}\cdot M^\ell$
and neither smoothing nor interpolation affects
the cost bound.

The adaptive MLMC algorithm that will be introduced in
Section \ref{sec:algorithm} involves, in particular,
a variance estimation step. It is crucial that the cost
for this step also stays within the cost bound \eqref{eq:cost:mlmc:new}.

\subsection{Error Decomposition}\label{err:dec}

Observe that $\M$ is square-integrable, since $g$ is bounded, 
so that (E2) yields
$\Exp\|Q_k^{r}(\M)\|_\infty^2 < \infty$. 
The error of $Q_k^{r}(\M)$ is defined by 
\begin{equation}\label{g48}
\err (Q_k^{r}(\M)) = \left( \Exp\|F - Q_k^{r}(\M)\|_\infty^2 \right)^{1/2}.
\end{equation}

The variance of any square-integrable $\R^k$-valued random 
variable $\M$ is defined by
\[
\V (\M) = \Exp|\M - \Exp(\M)|_\infty^2,
\]
and 
\[
\Exp|x - \M|_\infty^2 \leq 2 \cdot (|x - \Exp(\M)|_\infty^2 + \V(\M))
\]
holds for $x \in \R^k$. 

For $\delta > 0$ the error of $Q_k^{r}(\M)$ may be decomposed into
the interpolation error
\[
e_1 = \bigl\|F - Q_k^{r}(F(s_1), \dots, F(s_k))\bigr\|_\infty,
\]
the smoothing error
\[
e_2 = \bigl|(F(s_1), \dots, F(s_k)) - \Exp (g^{k,\delta} (Y))\bigr|_\infty,
\]
and the bias
\[
e_3 =
\bigl|\Exp (g^{k,\delta} (Y)) -\Exp (g^{k,\delta} (Y^{(L_1)}))
\bigr|_\infty 
\]
as well as the variance 
\[
e_4 = \V (\M)
\]
of $\M$. In fact, we have
\begin{align*}
\err(Q_k^{r}(\M)) &\leq 
e_1 + \left( \Exp \| Q_k^{r}( (F(s_1), \dots, F(s_k))-\M)
\|_\infty^2\right)^{1/2}\\
&\leq
e_1 + \|Q_k^r\| \cdot 
\left( \Exp | (F(s_1), \dots, F(s_k))-\M|_\infty^2\right)^{1/2},
\end{align*}
where $\|Q_k^r\|$ denotes the Lipschitz constant of 
$Q_k^r$.
Since
\begin{align*}
\left( \Exp | (F(s_1), \dots, F(s_k))-\M|_\infty^2\right)^{1/2}
&\leq 
e_2 +
\left( \Exp | \Exp(g^{k,\delta}(Y)) -\M|_\infty^2\right)^{1/2}\\
&\leq e_2 + \sqrt{2} \cdot (e^2_3+e_4)^{1/2},
\end{align*}
we obtain
\begin{equation}
\label{err:dec:num}
\err(Q_k^{r}(\M)) \leq 
e_1 + \|Q_k^r\| \cdot \left( e_2 + \sqrt{2} \cdot (e^2_3 + e_4)^{1/2}
\right).
\end{equation}
If we do not apply smoothing, i.e., for $\delta = 0$, we formally
have $e_2=0$, which leads to
\begin{equation}\label{g69}
\err(Q_k^{r}(\M)) \leq 
e_1 + 
\sqrt{2} \cdot \|Q_k^r\| \cdot 
\left(
|(F(s_1), \dots, F(s_k)) -\Exp (\M)|^2_\infty+ e_4\right)^{1/2}.
\end{equation}

The Bienaym\'e formula for real-valued random variables turns
into the inequality
\[
\V (\M)  \leq c(k) \cdot 
\sum_{i=1}^n \V(\M_i),
\]
if $\M = \sum_{i=1}^n \M_i$ with independent square-integrable random 
variables $\M_i$ taking values in $\R^k$.
Here $c(k)$ only depends on the dimension $k$, and
there exist $c_1,c_2 >0$ such that
\begin{equation}\label{eq300}
c_1 \cdot \log(k+1) \leq c(k) \leq c_2 \cdot \log(k+1)
\end{equation}
for every $k \in \N$.
In the context of multilevel algorithms this is exploited
in \cite{heinrich98}
for the first time.
We refer to \eqref{eq204} and \eqref{eq202} in the Supplementary Materials
for an explicit value of $c(k)$.

Consequently,
\begin{equation}\label{eq301}
\V (\M) \leq c(k) \cdot
\left( \frac{\V (g^{k,\delta}(Y^{(L_0)}))}{N_{L_0}} 
+ 
\sum_{\ell=L_0+1}^{L_1}
\frac{
\V (g^{k,\delta}(Y^{(\ell)}) - g^{k,\delta}(Y^{(\ell-1)}))}
{N_\ell} 
\right)
\end{equation}
for the variance of the multilevel algorithm $\M$. 

\section{Asymptotic Analysis of the MLMC Algorithm}\label{s3}

Using the improved cost bound \eqref{eq:cost:mlmc:new}
we derive an improved version of \cite[Thm.\ 2.6]{giles2014multi},
which gives an asymptotic upper bound of the cost of
the multilevel algorithm in terms of its error.

\subsection{Assumptions on Weak and Strong Convergence}\label{s2.2}

In our analysis of the bias and the variance of the
multilevel algorithm $\M$ we assume that the following properties
hold for some constant $c>0$ with $M$ according to (A2):
\begin{enumerate}
\item[(A3)] There exist constants
$\alpha_1 \geq 0$, $\alpha_2>0$, and $\alpha_2 \geq \alpha_3
\geq 0$
such that the weak error estimate
\[
\sup_{s \in [S_0,S_1]} \left| \Exp\left( 
g((Y- s)/\delta) - g((Y^{(\ell)} - s)/\delta) \right) \right|
\leq c \cdot \min \left(
\delta^{-\alpha_1} \cdot M^{-\ell \cdot \alpha_2}, M^{-\ell \cdot
\alpha_3} \right)
\]
holds for all $\delta \in \left]0,\delta_0\right]$ 
and $\ell \in \N_0$. 
\item[(A4)] There exist constants $\beta_4 \geq 0$ 
and $\beta_5 >0$ such that the strong error estimate
\[
\Exp\min ( (Y-Y^{(\ell)})^2/\delta^2, 1) 
\leq c \cdot 
\delta^{-\beta_4} \cdot M^{-\ell \cdot \beta_5}
\]
holds for all $\delta \in \left]0,\delta_0\right]$ and 
$\ell \in \N_0$. 
\end{enumerate}

Assumption (A4) and the Lipschitz continuity and boundedness of $g$
immediately yield the following fact,
see \cite[Lemma~2.4]{giles2014multi}.

\begin{lemma}\label{l2-xx}
There exists a constant $c>0$ such that
\[
\Exp\sup_{s \in [S_0,S_1]} 
\bigl( g((Y-s)/\delta) - g((\tl-s)/\delta)\bigr)^2
\leq c \cdot \min( \delta^{-\beta_4} \cdot M^{-\ell\cdot \beta_5},1)
\]
holds for all $\delta \in \left]0,\delta_0\right]$ and 
$\ell \in \N_0$.
\end{lemma}

\begin{remark}\label{r200}
For the analysis of the single-level algorithm
$\SSS^{k,\delta,L}_N$, i.e., for
$L=L_0=L_1,$ it suffices to
assume that
the simulation of the distribution of $Y^{(\ell)}$ is possible at cost
at most $c \cdot M^\ell$ for every $\ell \in \N_0$, cf.\ (A2). 
Furthermore, there is no need for a strong error estimate like
(A4), and if we do not employ smoothing, then
(A3) may be replaced by the following assumption.
There exist a constant $\alpha > 0$ 
such that the weak error estimate
\[
\sup_{s \in [S_0,S_1]} \left| \Exp\left( 
1_{\left]-\infty,s\right]} (Y) - 1_{\left]-\infty,s\right]} (Y^{(\ell)})
\right) \right|
\leq c \cdot M^{-\ell \cdot \alpha}
\]
holds for all $\ell \in \N_0$. 
It turns out that
the analysis of single-level algorithms  $\SSS^{k,0,L}_N$without smoothing
is formally reduced to the case $\delta >0$
if we take
\begin{equation}\label{g203}
\alpha_1 = 0, \qquad \alpha_2 = \alpha, \qquad \alpha _3 = \alpha.
\end{equation}
\end{remark}

\subsection{Main Result}\label{s3.2}

We say that a sequence of randomized algorithms $\A_n$ 
converges with order $(\gamma,\eta) \in \left]0,\infty\right[ 
\times \R$ if $\lim_{n\to \infty} \err (\A_n) = 0$ and if
there exists a constant $c > 0$ such that
\[
\cost (\A_n) \leq c \cdot (\err (\A_n))^{-\gamma} \cdot
( - \log \err (\A_n))^{\eta}.
\]
Moreover, we put
\[
q = \min \left( \frac{r+1+\alpha_1}{\alpha_2}, \frac{r+1}{\alpha_3}
\right).
\]

\begin{theorem}\label{t2}
Assume that the cost bound \eqref{eq:cost:mlmc:new} is satisfied.
The following order, with $\eta = 1$, is achieved
by algorithms $Q_k^{r}(\MM)$ with suitably chosen parameters:
\begin{align}
\label{z1}
q \leq \beta_4/\beta_5
\quad&\Rightarrow\quad
\gamma = 2+ \frac{q}{r+1},\\
\label{z2}
q > \beta_4/\beta_5 \ \wedge\  \beta_5 > 1 
\quad&\Rightarrow\quad
\gamma = 2+ \frac{\beta_4/\beta_5}{r+1},\\
\label{z4}
q > \beta_4/\beta_5 \ \wedge\  \beta_5 < 1 
\quad&\Rightarrow\quad
\gamma = 2+ 
\frac{\beta_4+(1-\beta_5)\cdot q}{r+1}. 
\intertext{Moreover, with $\eta =3$,}
\label{z5}
q > \beta_4 \ \wedge\  \beta_5 = 1
\quad&\Rightarrow\quad
\gamma = 2 + \frac{\beta_4}{r+1}. 
\end{align}
\end{theorem}

The proof of this result, which also includes the choice of the
parameters of the multilevel algorithm,
follows the one presented in 
\cite{giles2014multi} and can be seen in Section \ref{SM:prt2} in 
the Supplementary Materials.

Theorem \ref{t2} improves \cite[Thm.~2.6]{giles2014multi},
if \eqref{eq:cost:mlmc:new} is satisfied. This improved cost bound
also leads to improved versions of 
\cite[Thm.~3.3, Thm.~4.3]{giles2014multi}, which deal with the
approximation of densities on compact intervals and distribution
functions at a single point, respectively.

\begin{exmp}
Suppose that $Y = \varphi(X)$, where $X$ is a sufficiently smooth
Gaussian process (or random field) on a compact domain and
$\varphi$ is a Lipschitz continuous real-valued functional
with respect to the supremum
norm, say, the supremum norm itself. Using appropriate 
approximations $X^{(\ell)}$ of $X$ and putting
$Y^{(\ell)}=\varphi(X^{(\ell)})$ the estimate (5.1) in
\cite{giles2014multi} is satisfied with a large value of $\beta$
due to the smoothness of $X$.
It what follows, $\varepsilon > 0$ may be chosen arbitrarily small.

We obtain $\beta_4=2$ and $\beta_5=\beta$ as well as
$\alpha_1 = \varepsilon$, $\alpha_2=\beta/2$, and $\alpha_3 = \beta/2 -
\varepsilon$, see \cite[Sec.~5]{giles2014multi}. 
This leads to $q = 2(r+1)/\beta + \eps$, while $\beta_4/\beta_5 =
2/\beta$. Hence \eqref{z2} yields 
\[
\gamma = 2 + \frac{2}{\beta \cdot (r+1)}.
\]
Since $\max(1,\beta_4/\beta_5) = 1$, the case (2.10) in
\cite[Thm.~2.6]{giles2014multi} 
only yields
\[
\gamma = 2 + \frac{1}{r+1}.
\]
\end{exmp}

\begin{remark}
In the limit $r \to \infty$ we obtain
$\gamma = 2+\max(1-\beta_5,0)/\alpha_2$ in Theorem \ref{t2},
i.e., we recover the order of convergence for the
standard MLMC application, namely, the approximation of
expectations.
\end{remark}

\section{Sketch of the Adaptive MLMC Algorithm}\label{sec:algorithm}

We present an MLMC algorithm, which assumes that the parameters for the 
weak and the strong convergence are unknown and have to be
substituted by suitable estimates during the algorithm run. 
Based on these estimates we determine the replication numbers,
the range of levels, the smoothing parameter, and the number of
interpolation points adaptively. For simplicity, the minimal level
is chosen as $L_0=0$, and we use $L=L_1$ to denote the maximal
level.

We do not address the issue how to detect the smoothness $r$ of the 
density $\rho$ from simulation data, and how to choose an
interpolation scheme that exploits the smoothness in an optimal
way. Instead, we assume a known $r \geq 1$,
and we take a function $g$ with the properties (S1)--(S4).
Furthermore, we use piecewise polynomial
interpolation of degree $r$ at equidistant points.
The number $k_n$ of interpolation points is given by
\[
k_n = \lceil 2^n/r\rceil \cdot r + 1,
\]
where $n \in \N$ with $2^{n+1} > r$,
and the points themselves are given by
\[
s_j = s_{j,n} = S_0  + (j-1) \cdot (S_1-S_0)/(k_n-1)
\]
for $j=1,\dots,k_n$.
We use
\[
Q_n=Q^r_{k_n}
\]
to denote the piecewise polynomial interpolation of degree
$r$ at these points, together with the monotonicity corrections as 
described in Remark \ref{r3a}, and we put 
\[
F_n = (F(s_{1}),\dots,F(s_{k_n})).
\]
Obviously, (E1) and (E2) are satisfied, and the Lipschitz
constant $\|Q_n\|$ of $Q_n$ does not
depend on $n$. In the numerical experiments we take $r=3$ and $g$ 
according to Remark \ref{r2}.

The smoothing parameter $\delta$ is chosen from the discrete set of values 
\[
\delta_m = 
(S_1-S_0)/ 2^{m},
\]
where $m \in \N$. 
With a slight abuse of notation we put
\[
g^{n,m} = g^{k_n,\delta_m}.
\]

For a given $\Eps >0$ we wish to select the parameters
of the MLMC algorithm such that its error is at most $\Eps$
and its cost is as small as possible.
Our approach to the selection of the replication numbers and of 
the maximal level follows \cite{giles15}, 
who studies MLMC algorithms 
with values in $\R$. The latter is adapted to the present case of 
vector-valued algorithms and extended to also handle the
selection of the smoothing parameter and the number of
interpolation points.

In order to achieve 
\begin{equation}\label{eq222a}
\err(Q_n(\M)) \leq \Eps
\end{equation}
we have to assign certain proportions of $\Eps$ to the four sources of 
the error, which have been introduced in Section \ref{err:dec}.
The analysis presented in Theorem \ref{t2} yields an 
asymptotically optimal choice of the parameters of the
multilevel algorithm, in which case the
cost is asymptotically bounded by
$\sum_{\ell=0}^{L} N_{\ell}\cdot M^\ell$, see \eqref{f7},
and $k \cdot \delta$ is of the order one, see \eqref{f9a} and
\eqref{f9b}.
This suggests assigning only a small part of $\Eps$ to the
interpolation error $e_1$ and to the smoothing error $e_2$.
While $k$ only has an impact on the error and cost of 
$\M$ via a factor of order $\log(k)$, a small value of $\delta$ 
might harm the decay of the bias and the variances, see \eqref{g99}. 
Accordingly, we aim at $e_1$ being smaller than $e_2$.
Specifically we wish
to choose the parameters of our algorithm such
that
\begin{equation}\label{eq222}
e_1
\le \|Q_n\| \cdot \Eps_*, \quad
e_2
\le 4 \cdot \Eps_*, \quad
e_3
\le 16 \cdot \Eps_*, \quad
e_4 \le 256 \cdot \Eps_*^2,
\end{equation}
where
\[
\Eps_* = \frac{\Eps}{37\cdot\|Q_n\|}.
\]
By \eqref{err:dec:num} we get \eqref{eq222a}, if
\eqref{eq222} holds true. 
Recall that Remark \ref{r3} provides the explicit value
$\|Q_n\|\approx 1.63$ in the case $r=3$.

It is possible that a different assignment of $\Eps$ to the four
sources of the error may lead to better results,
numerically. Concerning the upper bounds \eqref{eq:cost:mlmc:new} 
and \eqref{g99} of the cost and the error, however, an 
improvement is possible at most by a multiplicative constant,
which does not depend on $\Eps$.
In the standard setting of MLMC to compute expectations, 
this assignment problem is further analyzed and a new
algorithm is constructed in \cite{collier_et_al2015}.

The present stage of the MLMC algorithm is defined, in particular,
by the parameter values $n$ and $m$ for interpolation and
smoothing, and the values of the maximal level $L$ and of the
replication numbers $N_\ell$. We always have
$L \geq 2$ and $N_\ell \geq 100$ for $\ell=0,\dots,L$. 
By the latter we ensure a reasonable accuracy in certain estimates
to be introduced below.
We use $y_{i,0}$ to denote samples of the random variable $Y^{(0)}$
and $(y_{i,\ell},y_{i,\ell-1})$ to denote samples of the random vector
$(Y^{(\ell)},Y^{(\ell-1)})$ for $\ell=1,\dots,L$. During the computation
the values of $n$, $m$, $L$, and $N_\ell$ 
are updated and all samples are stored.

\subsection{Assumptions}\label{s55}

For sequences of real numbers $u_\ell$ and positive real numbers
$w_\ell$ we write $u_\ell \approx w_\ell$, if
\[
\lim_{\ell \to \infty} u_{\ell}/ w_{\ell} = 1,
\]
and $u_\ell \ls w_\ell$, if
\[
\limsup_{\ell \to \infty} u_{\ell}/ w_{\ell} \leq 1.
\]

Assumption (A2) on the computational cost of simulating
the joint distribution of $Y^{(\ell)}$ and $Y^{(\ell-1)}$ is
assumed to hold. 
Assumption (A3) on the weak convergence, which is used in our 
asymptotic analysis, is replaced as follows.
For every $n$ and $m$, we suppose that there exists $c,\alpha > 0$
such that
\begin{equation}\label{eq206}
|\Exp (g^{n,m} (Y^{(\ell)})) -\Exp (g^{n,m} (Y^{(\ell-1)})) |_\infty
\approx c \cdot M^{-\ell\cdot\alpha}.
\end{equation}
Furthermore, we assume that 
\begin{equation}\label{eq206b}
\lim_{\ell \to \infty} \Exp (g^{n,m} (Y^{(\ell)})) = 
\Exp (g^{n,m} (Y)).
\end{equation}
This yields the asymptotic upper bound
\begin{align}\label{eq206a}
&\bigl|\Exp (g^{n,m} (Y)) -\Exp (g^{n,m} (Y^{(\ell)})) \bigr|_\infty 
\notag \\
& \qquad  \ls (M^\alpha-1)^{-1} \cdot
\bigl|\Exp (g^{n,m} (Y^{(\ell)})) -\Exp (g^{n,m} (Y^{(\ell-1)}))
\bigr|_\infty
\end{align}
for the bias at level $\ell$.
In contrast to our asymptotic analysis, which makes
use of (A4), the construction of the adaptive MLMC
algorithm is not based on any assumption on the strong convergence.

Put
\[
C_r = 2^{r+1}.
\]
For every $n$, we suppose that there exists $c >0$ such that
\begin{equation}\label{eq211}
\bigl|\Exp (g^{n,m} (Y)) -\Exp (g^{n,m-1} (Y))
\bigr|_\infty
\approx c \cdot \delta_m^{r+1}.
\end{equation}
This yields the asymptotic upper bound 
\begin{equation}\label{eq211a}
\bigl|F_n -\Exp (g^{n,m} (Y)) \bigr|_\infty
\ls (C_r-1)^{-1} \cdot
\bigl|\Exp (g^{n,m} (Y)) -\Exp (g^{n,m-1} (Y))\bigr|_\infty
\end{equation}
for the smoothing error with parameter $\delta_m$.

We suppose that there exists $c >0$ 
such that
\begin{equation}\label{eq220}
\| Q_n(F_n) - Q_{n-1}(F_{n-1})\|_\infty
\approx c \cdot C_r^{-n}.
\end{equation}
This yields the asymptotic upper bound 
\begin{equation}\label{eq220a}
\| F - Q_n(F_n)\|_\infty
\ls (C_r-1)^{-1} \cdot
\| Q_n(F_n) - Q_{n-1}(F_{n-1})\|_\infty
\end{equation}
for the interpolation error with $k_n$ equidistant points. 

Formally, (A1) is not assumed to hold for the chosen value of $r$,
but of course the convergence order $r+1$ in the assumptions 
\eqref{eq211} and \eqref{eq220} corresponds to 
$\rho$ being at least $r$ times continuously differentiable.

\begin{remark}
Let $[z]_i$ denote the $i$-th component of $z \in \R^k$.
In the context of SDEs weak error results of the form
\begin{equation}\label{eq207}
\bigl[\Exp (g^{n,m} (Y)) -\Exp (g^{n,m} (Y^{(\ell)}))\bigl]_i
\approx c_i \cdot M^{-\ell \cdot \alpha} 
\end{equation}
with $c_i \neq 0$ and $\alpha > 0$
are known to hold for fixed $n$ and $m$ and all $i=1,\dots,k_n$ under 
suitable assumptions. Note that \eqref{eq207} implies \eqref{eq206b} 
as well as \eqref{eq206} with
\[
c = (M^\alpha -1) \cdot |(c_1,\dots,c_{k_n})|_\infty.
\]

Suppose that (A1) is satisfied for the chosen value of $r$ and that
$\rho^{(r)} (s_j) \neq 0$ for $j=1,\dots,k_n$ as well as 
\[
\int_{-1}^1 u^r \cdot (1_{\left]-\infty,0\right]} (u) - g(u)) \, du 
\neq 0.
\]
From the proof of Lemma \ref{l1} we get 
\begin{equation}\label{eq210}
\bigl[F_n - \Exp (g^{n,m} (Y))  \bigr]_i
\approx c_i \cdot \delta_m^{r+1} 
\end{equation}
with $c_i \neq 0$
for fixed $n$ and all $i=1,\dots,k_n$.
Note that \eqref{eq210} implies \eqref{eq211} with
\[
c = (C_r -1) \cdot |(c_1,\dots,c_{k_n})|_\infty.
\]
\end{remark}

Roughly speaking, our MLMC algorithm is based on the following
heuristics: the asymptotic bounds \eqref{eq206a}, \eqref{eq211a},
and \eqref{eq220a} are replaced by the corresponding inequalities,
and estimators for means and variances are assumed to be nearly exact.

\subsection{Variance Estimation and Selection of the Replication
Numbers}\label{ve}

The innermost loop of the algorithm performs, in particular,
a variance estimation with fixed parameters $n$, $m$, and $L$. 
It is important for the overall performance of the adaptive algorithm
that the cost for the variance estimation is of the order 
$O(k_n + c(N_0,\dots,N_L))$, where
\[
c(n_0,\dots,n_L) = 
\sum_{\ell=0}^{L} n_\ell \cdot (M^\ell + k_n \cdot \delta_m),
\]
cf.\ \eqref{eq:cost:mlmc:new}.  
We stress that there is no such constraint needed in the standard
MLMC application, which only involves expectations and variances
of real-valued random variables.

At first we estimate the expectation and the variance of the 
random vectors $g^{n,m}(Y^{(0)})$ and 
$g^{n,m}(Y^{(\ell)}) - g^{n,m}(Y^{(\ell-1)})$ for $\ell = 1, \dots, L$. 
To estimate the expectations we employ
\[
\hat{b}_0 = \frac{1}{N_0} \cdot \sum_{i=1}^{N_0} 
g^{n,m}(y_{i,0})
\]
and
\begin{equation}\label{eq205}
\hat{b}_\ell = \frac{1}{N_\ell} \cdot \sum_{i=1}^{N_\ell} 
(g^{n,m}(y_{i,\ell}) - g^{n,m}(y_{i,\ell-1})).
\end{equation}
As we have shown in Section \ref{s2.4},
the total cost for this step does not exceed the cost
bound $O(k_n + c(N_0,\dots,N_L))$.
To estimate the variances 
we cannot afford to use the whole data set.
Instead, we only use 
\[
N_\ell^\prime = \min\bigl(N_\ell, \max(\zeta, N_\ell \cdot
M^\ell)/k_n\bigr)
\]
samples on level $\ell$, where
\[
\zeta = \frac{1}{L} \cdot (k_n+ c(N_0,\dots,N_L)).
\] 
Accordingly, the variances are estimated by 
\begin{equation}\label{eq246}
\hat{v}_0 = \frac{1}{N^\prime_0} \cdot \sum_{i=1}^{N^\prime_0} 
|g^{n,m}(y_{i,0}) - \hat{b}_0|_\infty^2
\end{equation}
and
\begin{equation}\label{eq245}
\hat{v}_\ell = \frac{1}{N^\prime_\ell} \cdot \sum_{i=1}^{N^\prime_\ell} 
|g^{n,m}(y_{i,\ell}) - g^{n,m}(y_{i,\ell-1})
 - \hat{b}_\ell|_\infty^2.
\end{equation}
Since $\hat{v}_\ell$ can be computed at cost
$O(\max(\zeta,N_\ell\cdot M^\ell))$
for $\ell=0,\dots,L$, the total cost for this step 
also stays within the cost bound 
$O(k_n + c(N_0,\dots,N_L))$.
Obviously $\zeta \geq N_0/L$, and $N_0/L$ is of
the order $\Eps^{-2}$, while $k_n$ is of the order
$\Eps^{-1/(r+1)}$, cf.\ \eqref{f9b}, \eqref{f9d}, and
\eqref{f9c}. Hence $\zeta/k_n$ tends to infinity as $\Eps \to 0$.
On the other hand, $M^L$ is of the order $\Eps^{-q/(r+1)}$, 
and therefore $M^L/k_n$ tends to infinity as $\Eps \to 0$ if $q >
1$.

With $c(k)$ given by \eqref{eq202} from the Supplementary Materials, 
\[
\hat{v}(n_0,\dots,n_L) = c(k_n) \cdot 
\sum_{\ell=0}^L \frac{1}{n_\ell} \cdot \hat{v}_\ell 
\]
serves as an empirical upper bound for the variance of
the MLMC algorithm with any choice of replication numbers $n_\ell$.
If, for the present choice of replication numbers, 
this bound is too large
compared to the upper bound for $\V(\M)$ in \eqref{eq222},
i.e., if the variance constraint
\begin{equation}\label{eq203}
\hat{v}(N_0,\dots,N_{L}) 
\leq 
256 \cdot \Eps_*^2
\end{equation}
is violated,
we determine new values of $N_0,\dots,N_L$ by minimizing 
$c(n_0,\dots,n_L)$ subject to the constraint
$\hat{v}(n_0,\dots,n_{L}) \leq 256 \cdot \Eps_*^2$
This leads to
\begin{equation}\label{eq247}
n_\ell
= \frac{\hat{v}_\ell^{1/2}}{(M^\ell + k_n \cdot \delta_m)^{1/2}}
\cdot \sum_{\ell=0}^L 
\left( \hat{v}_\ell \cdot (M^\ell + k_n \cdot
\delta_m)\right)^{1/2} \cdot 
\frac{c(k_n)}{256}\cdot \Eps_*^{-2},
\end{equation}
and extra samples of $Y^{(0)}$ and
$(Y^{(\ell)},Y^{(\ell-1)})$ have to be generated accordingly.
Note that the updated estimates for the expectations and variances can be 
computed within the updated cost bound.

\subsection{Bias Estimation and Selection of the Maximal Level}\label{be}

For fixed $n$ and $m$,
we wish to determine the smallest value of $L$ such that 
\[
\bigl|\Exp (g^{n,m} (Y^{(L)})) -\Exp (g^{n,m} (Y^{(L-1)})) \bigr|_\infty
\leq 16 \cdot (M^\alpha-1) \cdot \Eps_*
\]
is satisfied, which corresponds 
to the upper bound for $e_3$ in \eqref{eq222} together
with \eqref{eq206a}. Initially we try $L=2$. 

For $|\Exp (g^{n,m} (Y^{(\ell)})) -\Exp
(g^{n,m} (Y^{(\ell-1)})) |_\infty$
the estimate $|\hat{b}_\ell|_\infty$ is 
available on the levels $\ell=1,\dots,L$, see \eqref{eq205}.
To ensure a reasonable accuracy in these estimates
we only consider those levels, 
where the replication number $N_\ell$ exceeds a certain threshold,
and we let ${\cal L }$ denote the corresponding subset of
$\{1,\dots,L\}$; in the numerical
experiments we choose $10^4$ as the threshold value.
We estimate $\alpha$ and $c$ in \eqref{eq206} by a least-squares fit,
i.e., we take $\hat{\alpha}$ and $\hat{c}$ to minimize 
\begin{equation}
\label{eq:regression}
(\alpha,c) \mapsto
\sum_{\ell\in\cal{L}} 
\left( \log |\hat{b}_{\ell}|_\infty + 
\ell \cdot \alpha \, \log M+ \log c \right)^2.
\end{equation}

While the value of $\hat{c}$ is irrelevant, 
an upper bound for the norm of
$\Exp (g^{n,m} (Y^{(L)})) -\Exp
(g^{n,m} (Y^{(L-1)}))$
is given by
$|\hat{b}_L|_\infty$, 
or, more generally, by
$M^{(\ell-L) \cdot \hat{\alpha}} 
\cdot |\hat{b}_\ell|_\infty$ with $\ell \leq L$.
Hence we put
\begin{equation}\label{eq238}
\hat{B}_2 =
\max \bigl(
|\hat{b}_2|_\infty,
|\hat{b}_{1}|_\infty/M^{\hat{\alpha}}
\bigr)
\end{equation}
for $L=2$ and
\begin{equation}\label{eq239}
\hat{B}_L =
\max \bigl(
|\hat{b}_L|_\infty,
|\hat{b}_{L-1}|_\infty/M^{\hat{\alpha}},
|\hat{b}_{L-2}|_\infty/M^{2 \hat{\alpha}}
\bigr)
\end{equation}
for $L \geq 3$.

The present value of $L$ is accepted as the maximal level, 
if 
the bias constraint
\begin{equation}
\label{bias:stop}
\hat{B}_L 
\leq 
16 \cdot (M^{\hat{\alpha}}-1) \cdot \Eps_*
\end{equation}
is satisfied.
Otherwise, $L$ is increased by one, and new 
samples will be generated. As already mentioned,
we take $N_L=100$ as a default value.

In our simulations we will ignore the cost for the performing the 
regression \eqref{eq:regression}, as its cost is absolutely negligible, 
compared to number of other operations, performed by the algorithm.

\subsection{Selection of the Smoothing Parameter}\label{sp}

Our approach to choose the smoothing parameter
closely follows the approach for the selection of the maximal level,
as we consider two values, $\delta_m$ and $\delta_{m-1}$, of the
smoothing parameter at the same time. 

The parameter $n$, which determines the number $k_n$ of interpolation
points, is fixed.
We wish to determine the smallest value of $m$, i.e., the largest
value of $\delta_m$, such that 
\[
\bigl|\Exp (g^{n,m} (Y)) -\Exp (g^{n,m-1} (Y)) \bigr|_\infty
\leq 
4 \cdot (C_r-1) \cdot \Eps_*
\]
is satisfied, which corresponds to 
the upper bound for $e_2$ in \eqref{eq222} together with \eqref{eq211a}.
Initially we try $m=2$. 
Actually, $Y$ is approximated by $Y^{(L)}$,
and an upper bound for 
$|\Exp (g^{n,m} (Y^{(L)})) -
\Exp (g^{n,m-1} (Y^{(L)}))|_\infty$ 
is given by
\begin{equation}\label{eq234}
\hat{s} = 
\Bigl| \frac{1}{N_L} \cdot \sum_{i=1}^{N_L} 
(g^{n,m} (y_{i,L}) - g^{n,m-1} (y_{i,L}))\Bigr|_\infty.
\end{equation}
The present value of $\delta_m$ is accepted as the smoothing
parameter, if the smoothing constraint
\begin{equation}\label{eq212}
\hat{s} \leq 4 \cdot (C_r-1) \cdot \Eps_*
\end{equation}
is satisfied.
Otherwise, $m$ is increased by one. Due to the update for the smoothing 
parameter $\delta$, we need to update all the estimates for the bias and the variance, which increases the overall cost of the algorithm by 
$O(k_n +\max(k_n\cdot \delta_m,1)\cdot\sum_{\ell=0}^L N_\ell)$.

Alternatively, the explicit
error bound from Remark \ref{r1} may be used to select the
smoothing parameter, if an upper bound or reliable estimate
for $|\rho^{(3)}|$ is available.

\subsection{Selection of the Number of Interpolation Points}\label{ip}

The procedure to choose the number $k_n$ of interpolation points
mimics our approach to choose the smoothing parameter, i.e.,
we consider two interpolation schemes, $Q_n$ and $Q_{n-1}$, at the same
time. 

We wish to determine the smallest value of $n$, i.e., the smallest
number $k_n$ of interpolation points, such that
\[
\| Q_n(F_n) - Q_{n-1}(F_{n-1})\|_\infty \leq 
\|Q_n\| \cdot (C_r-1) \cdot \Eps_*,
\]
which corresponds to the upper bound for $e_1$ in \eqref{eq222}
together with \eqref{eq220a}.
Initially we try $n=2$.
Actually, $F_n$ and $F_{n-1}$ are approximated by 
$\Exp (g^{n,m} (Y^{(L)}))$ and $\Exp (g^{n-1,m} (Y^{(L)}))$,
respectively, and an upper bound
for the norm of
$Q_n(\Exp (g^{n,m} (Y^{(L)}))) -Q_{n-1}(\Exp (g^{n-1,m} (Y^{(L)})))$
is given by
\begin{equation}\label{eq233}
\hat{i} = 
\Bigl\| Q_n \Bigl(
\frac{1}{N_L} \cdot \sum_{i=1}^{N_L} g^{n,m} (y_{i,L})\Bigr) - 
Q_{n-1} \Bigl(
\frac{1}{N_L} \cdot \sum_{i=1}^{N_L} g^{n-1,m} (y_{i,L})\Bigr) 
\Bigr\|_\infty.
\end{equation}
The present value $k_n$ is accepted as the number of
interpolation points, if the interpolation constraint 
\begin{equation}\label{eq221}
\hat{i} \leq 
\|Q_n\| \cdot (C_r-1) \cdot \Eps_*
\end{equation}
is satisfied.
Otherwise, $n$ is increased by one. Again, as in Section \ref{sp}, we 
need to include the cost for evaluating the estimator at new points, 
and the added cost is of order 
$O(k_n +\max(k_n\cdot \delta_m,1)\cdot\sum_{\ell=0}^L N_\ell)$.

Alternatively, the explicit error bound from
Remark \ref{r3} may be used to select the
number of interpolation points, if an upper bound or reliable estimate
for $|\rho^{(3)}|$ is available.

\subsection{The Algorithm}\label{alg}

The desired accuracy $\Eps$ is the input to our MLMC algorithm.
\bigskip

\begin{itemize}
\item[] $n=1$;
\item[] $m=2$;
\item[] $L=2$;
\item[] $N_0=N_1=N_2=10^2$;
\item[] Generate these numbers of samples of $Y^{(0)}$ and 
$(Y^{(\ell)},Y^{(\ell-1)})$ for $\ell=1,2$;
\item[] Compute $\hat{v}_0,\hat{v}_1,\hat{v}_2$, 
see \eqref{eq246} and \eqref{eq245};
\item[] {\bf repeat} \quad /* interpolation */
\begin{itemize}
\item[] $n = n+1$;
\item[] $m = m-1$;
\item[] {\bf repeat} \quad /* smoothing */
\begin{itemize}
\item[] $m = m+1$;
\item[] $newlevel = {\rm false}$;
\item[] {\bf repeat} \quad /* bias */
\begin{itemize}
\item[] {\bf if} $newlevel$ {\bf then}
\item[] \qquad $L = L+1$;
\item[] \qquad $N_L = 100$;
\item[] \qquad Generate this number of samples of
$(Y^{(L)},Y^{(L-1)})$;
\item[] \qquad Compute $\hat{v}_L$, see \eqref{eq245};
\item[] {\bf endif;}
\item[] {\bf repeat}
\item[] \qquad Compute $n_0,\dots,n_L$, see \eqref{eq247};
\item[] \qquad 
$N_\ell = \max(N_\ell,n_\ell)$ for $\ell=0,\ldots,L$;
\item[] \qquad Generate extra samples of $Y^{(0)}$ and 
$(Y^{(\ell)},Y^{(\ell-1)})$ for $\ell=1,\ldots,L$\\
\mbox{}\qquad%
as needed;
\item[] \qquad Compute $\hat{v}_0,\dots,\hat{v}_L$, see 
\eqref{eq246} and \eqref{eq245};
\item[] {\bf until} the variance constraint \eqref{eq203} is 
satisfied;
\item[] Compute $\hat{\alpha}$, see \eqref{eq:regression},
and $\hat{B}_{L}$, see \eqref{eq238} and \eqref{eq239};
\item[] $newlevel = {\rm true}$;
\end{itemize}
\item[] {\bf until} 
the bias constraint \eqref{bias:stop} is satisfied;
 
\item[] Compute $\hat{s}$, see \eqref{eq234};
\end{itemize}
\item[] {\bf until}
the smoothing constraint \eqref{eq212} is satisfied; 
\item[] Compute $\hat{i}$, see \eqref{eq233};
\end{itemize}
\item[] {\bf until}
the interpolation constraint \eqref{eq221} is satisfied; 
\item[]
Compute $Q_n (\M^{k_n,\delta_m,0,L}_{N_{0}, \dots, N_{L}})$;
\end{itemize}

\bigskip

We comment on the cost for the individual steps of the algorithm
and their contributions to the overall cost.
By assumption (A2), samples of $Y^{(0)}$ or $(Y^{(\ell)},Y^{(\ell-1)})$ 
can be generated at cost $O(1)$ or $O(M^\ell)$, respectively.
For the present values of $L$, $n$, $m$, and
$N_0,\dots,N_L$ the variance estimates $\hat{v}_0,\dots,\hat{v}_L$
can be compute at cost $O(k_n + c(N_0,\dots,N_L))$, see Section
\ref{ve}.
Moreover, these two steps are the dominating ones, i.e., the cost for all 
other steps of the algorithm is negligible. Actually, the whole bias 
loop with terminal values $L$ and $N_0,\dots,N_L$ can be executed at 
cost $O(k_n + c(N_0,\dots,N_L))$, see Section \ref{ve}.
Summing up these quantities over all iterations of the interpolation
and the smoothing loops yields a bound for the overall cost
of the algorithm including an additional cost for the updating the 
smoothing coefficient, see Section \ref{sp}, and number of interpolation 
points, see Section \ref{ip}. In the numerical experiments, which
are presented in the following section, the constant in the
$O$-notation is taken to be one.

\section{Numerical Experiments}\label{s6}
In this section we will apply our adaptive general approach for 
approximating the distribution function, based on the multilevel Monte 
Carlo approach, in the context of stochastic differential equations. We 
would like to point at the fact, that all the previous presentation 
does not assume in any way that we work in the SDE context.

We consider three benchmark problems for a simple, scalar SDE, where the
solutions are known analytically. Our numerical experiments show the 
computational gain in terms of upper bounds, 
achieved by the adaptive multilevel Monte Carlo approach with smoothing
in comparison to a non-adaptive
single-level Monte Carlo approach without smoothing. 
Furthermore, we compare the error of the multilevel algorithm
with the accuracy demand $\Eps$, which serves as an input to
the algorithm.

Consider a geometric Brownian motion $X$, given by 
\begin{equation}
\label{eq:gbm}
\begin{aligned}
\phantom{\qquad t \in [0,T],}
dX_t & = \mu\cdot X_t\, dt + \sigma\cdot X_t\, dW_t, 
\qquad t \in [0,T],\\
X_0 &= 1,
\end{aligned}
\end{equation}
where $W$ denotes a scalar Brownian motion.
For the approximation of $X$ we use the Milstein scheme
with equidistant time-steps $h_\ell=M^{-\ell}\, T$ 
with $M=2$ 
and with piecewise linear interpolation between the interpolation
points.
 
Given $\Eps$, we use the algorithm from Section \ref{alg}, which 
estimates all the necessary parameters on the fly, so no prior knowledge 
of the convergence properties of the discretization scheme is needed. 
The cost of an individual run of
the algorithm is calculated as described in Section \ref{alg} 
and includes 
path generation and variance and bias estimation 
(Sections \ref{ve} and \ref{be}),
functional evaluations due to the updates of the 
smoothing coefficients (Section \ref{sp}) and interpolation points
(Section \ref{ip}). The only costs we do not include, 
since they are negligible, are the 
costs for the regression \eqref{eq:regression} and the 
monotonicity corrections (Remark \ref{r3a}).

The cost as well as the error, i.e., the supremum norm distance between 
the output and the true distribution function $F$, 
for an individual run of the adaptive multilevel algorithm
are random quantities, which depend on $\Eps$. By taking
expectations we get two deterministic quantities that characterize the
performance of the multilevel algorithm. More precisely, we
consider the root mean squared
error, cf.~\eqref{g48}, which will be denoted by
$\err_{\text{ML}}(\Eps)$,
and the expected cost, which will be denoted by $\cost_{\text{ML}}(\Eps)$.

Since $\err_{\text{ML}}(\Eps)$ and
$\cost_{\text{ML}}(\Eps)$ are not known exactly, we employ a simple Monte 
Carlo algorithm with $100$ independent replications for each of 
the values $\Eps = 2^{-i},\ i=3,\ldots,9$. 
The corresponding empirical means are denoted by 
$\hat{e}_{\text{ML}} (\Eps)$ and $\hat{c}_{\text{ML}}(\Eps)$, respectively.

To assess the accuracy of the multilevel algorithm,
$\err_{\text{ML}}(\Eps)$ should be compared with 
the desired accuracy $\Eps$. Our present approach provides control of
the error of the multilevel algorithm
for a given $\Eps$, therefore we aim at
$\hat{e}_{\text{ML}}(\Eps)$ being less than $\Eps$.

To specify the computational gain we need to 
choose the parameters $k_n$, $L$, and $N$
of a single-level Monte Carlo method $\mathcal{S}$ without smoothing
in a fair way.
As we have discussed in Section \ref{err:dec}, if we do not apply 
smoothing, i.e., for $\delta = 0$, we formally have $e_2=0$, which leads 
to 
\[
\err(Q_n(\mathcal{\mathcal{S}})) \leq 
e_1 + 
\sqrt{2} \cdot \|Q_n\| \cdot 
\left(
|(F(s_1), \dots, F(s_k)) -\Exp (\mathcal{S})|^2_\infty+
\V(\mathcal{S})\right)^{1/2},
\]
cf.~\eqref{g69}.
Hence we aim at
\[
e_1 \le \|Q_n\|\cdot\Eps_*, \quad
|(F(s_1), \dots, F(s_k)) -\Exp (\mathcal{S})|_\infty
\le 16\cdot\Eps_*, \quad
\V (\mathcal{S})
\le 256\cdot\Eps_*^2,
\]
where $\Eps_*=\Eps/(33\cdot \|Q_n\|)$ in this case, again with 
$\|Q_n\|\approx 1.63$.
We choose 
\[
k_n = (\|Q_n\|\cdot\Eps_*)^{-1/4},
\]
up to the appropriate rounding,
which corresponds to the assumption that (E3)
holds with a constant $c$ close to one. 
Moreover, based on \eqref{eq301} with $L=L_0=L_1$, with $\log_2 k_n$ 
instead of
$c(k)$, and with the assumption that $\V(g^{k_n,0}(Y^{(L)}))$ is
approximately one, we  take 
$(\log_2 k_n)/(256 \cdot \Eps_*^{2})$
replications, up to integer rounding, in the single-level algorithm.
Finally, we assume that the weak
error, as considered in Remark \ref{r200},
is bounded from above by $h^{\hat{\alpha}}$, where
$h$ denotes the step-size of the Milstein scheme.
The exponent $\hat{\alpha}$ is estimated empirically and
provided to the single-level algorithm,
but the cost for this is not taken into account, since we 
consider
\[
\cost_{\text{SL}} (\Eps) = 
\log_2 k_n\cdot \frac{1}{256\cdot\epsilon_*^2}\cdot 
\left(k_n + (16\cdot\epsilon_*)^{-1/\hat{\alpha}}\right)
\]
as the cost for the single-level Monte Carlo algorithm.

The ratio 
$\cost_{\text{SL}}(\Eps) / \cost_{\text{ML}} (\Eps)$
determines how effective is our approach,
and consequently we will use
$\cost_{\text{SL}}(\Eps) / \hat{c}_{\text{ML}} (\Eps)$
to specify the computational gain.

\subsection{Smooth Path-independent Functionals for SDEs}\label{sec:term}

In this section we set $\mu=0.05$, $\sigma=0.2$, and
$T=1$ in \eqref{eq:gbm}, and we approximate the distribution function 
of $Y = X_T$ 
on the interval $[S_0,S_1] = [0.5,1.5]$. Note that $Y$ is lognormally 
distributed with parameters $\mu-\sigma^2/2$ and $\sigma^2$.

The variance and the mean decay
with respect to the level $\ell$ for different values of $\delta$
and $7$ equidistantly placed points on $[S_0,S_1]$, 
along with the corresponding quantities
of the MLMC algorithm for the indicator function ($\delta=0$), 
are estimated based on $10^6$ samples and presented in the 
top two plots in Figure \ref{fig:europ}. 
The empirical values for the order of convergence
are close to $2$ for the variance if $\delta >0$,
close to $1$ for the variance if $\delta =0$ 
(see dashed reference lines for 
first and for second order decay), and
close to $1$ for the mean for any given $\delta$ 
(see dashed reference line for first order decay). 
We stress that this part of the numerical experiments
is not a part of our adaptive MLMC algorithm, and, of course,
the findings are not provided to the adaptive algorithm.

In the middle left plot in Figure \ref{fig:europ} we compare
the estimate $\hat{e}_{\text{ML}}(\Eps)$ (solid line) for the root mean 
squared error of the multilevel algorithm and the accuracy demand
$\Eps$ (dotted line). 
The computational gain over the single-level algorithm
can be seen in the middle right plot in Figure \ref{fig:europ}.

For the adaptive algorithm the number $k_n$ of interpolation points 
and the smoothing parameter $\delta_m$ are random quantities, which
are updated during the algorithm run and which depend on $\Eps$. 
Empirical means of the final values of $k_n$ and of the reciprocal of 
$\delta_m$, based on $100$ samples,
are presented in the bottom left plot in Figure \ref{fig:europ}. 

Finally, we show the true distribution function
on the interval $[S_0,S_1]$ (dashed line) along with 
two approximations at different accuracies: 
$\Eps = 2^{-3}$ (red line) and $\Eps = 2^{-9}$ 
(green line).

\begin{figure}
        \centering
        \includegraphics[width=\textwidth]{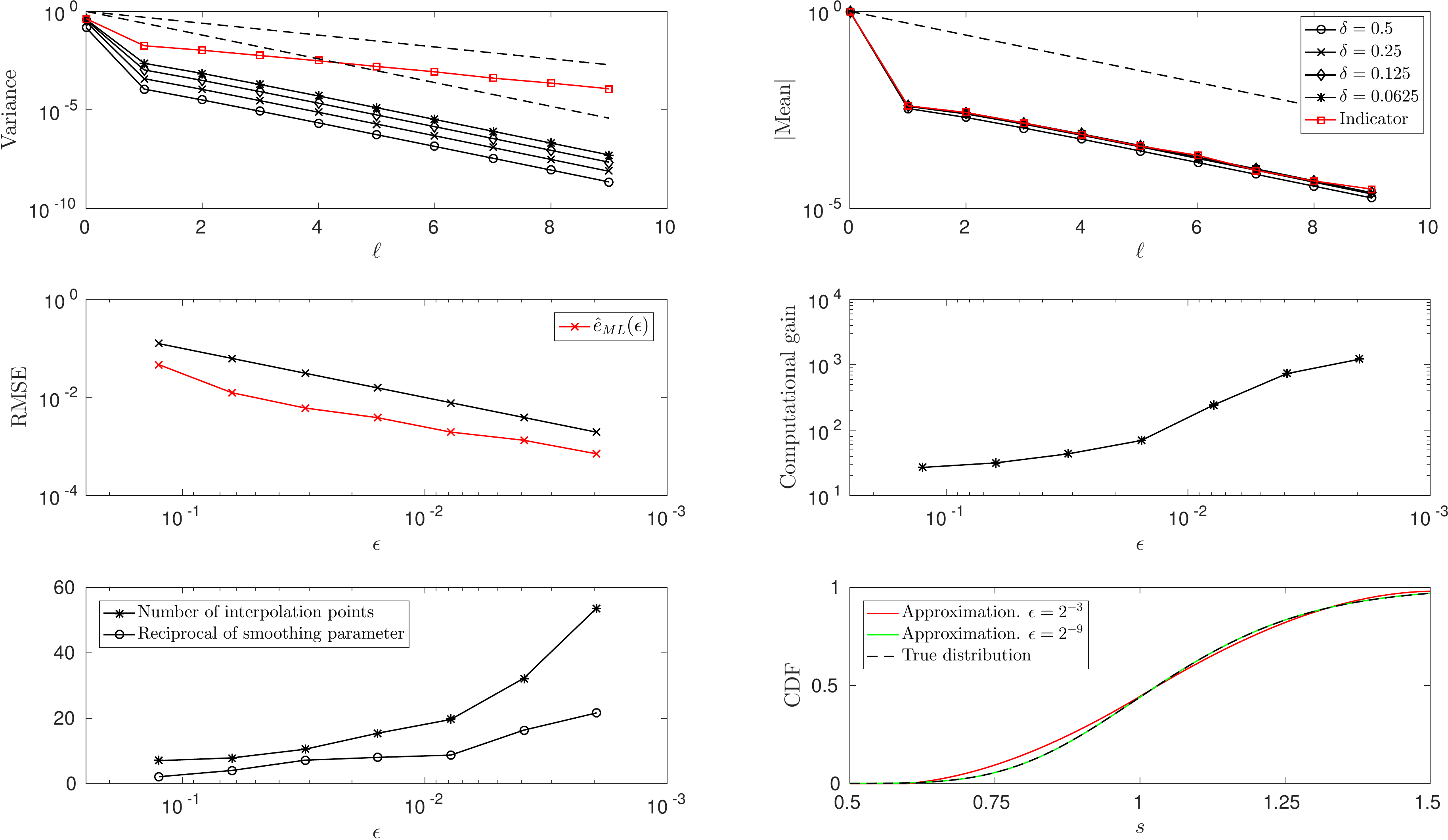}
\caption{Path-independent functional:
variance decay (top left), mean decay (top right), computational gain (middle right), RMSE (middle left), average smoothing coefficient and number of interpolation points (bottom left), true CDF and estimated for $\varepsilon^{-3}$ and $\varepsilon^{-9}$ CDFs.}
\label{fig:europ}
\end{figure}

\subsection{Smooth Path-dependent Functionals for SDEs}\label{sec:dep}

Consider the SDE \eqref{eq:gbm} with parameters $\mu=0.5$,
$\sigma=0.2$, and $T=1$. In this section we 
approximate the distribution function of
$Y = \max_{t\in[0,T]}X_t$
on the interval $[S_0,S_1] = [1.05,2.05]$. We have an explicit 
solution, see \cite{borodin2015handbook}, also in this case, since 
$$
F(s)
=1-\frac12\mathrm{erfc}\left(d_1\right) -\frac12\mathrm{erfc}
\left(d_2\right) s^{2\cdot\mu/\sigma^2-1}$$ 
with
\begin{gather*}
d_1=\frac{\ln(s)-(\mu-\sigma^2/2)\cdot T}
{\sigma\cdot\sqrt{2\cdot T}},\ 
d_2=\frac{\ln(s)+(\mu-\sigma^2/2)\cdot T}
{\sigma\cdot\sqrt{2\cdot T}}.
\end{gather*}

The plots in Figure \ref{fig:barrier} are obtained and
organised in the same way as the plots in Figure \ref{fig:europ}.
The empirical values for the order of convergence
are close to $0.85$ for the variance if $\delta >0$,
close to $0.5$ for the variance if $\delta =0$ (see dashed 
reference line for first order decay), and
close to $0.57$ for the mean for any given $\delta$ (see dashed reference 
line for half order decay). 
The true distribution function is shown on the interval $[S_0,S_1]$ 
along with two different approximations at different accuracies: 
$\Eps = 2^{-3}$ (red line) and $\Eps = 2^{-9}$ 
(green line). 

\begin{figure}
        \centering
        \includegraphics[width=\textwidth]{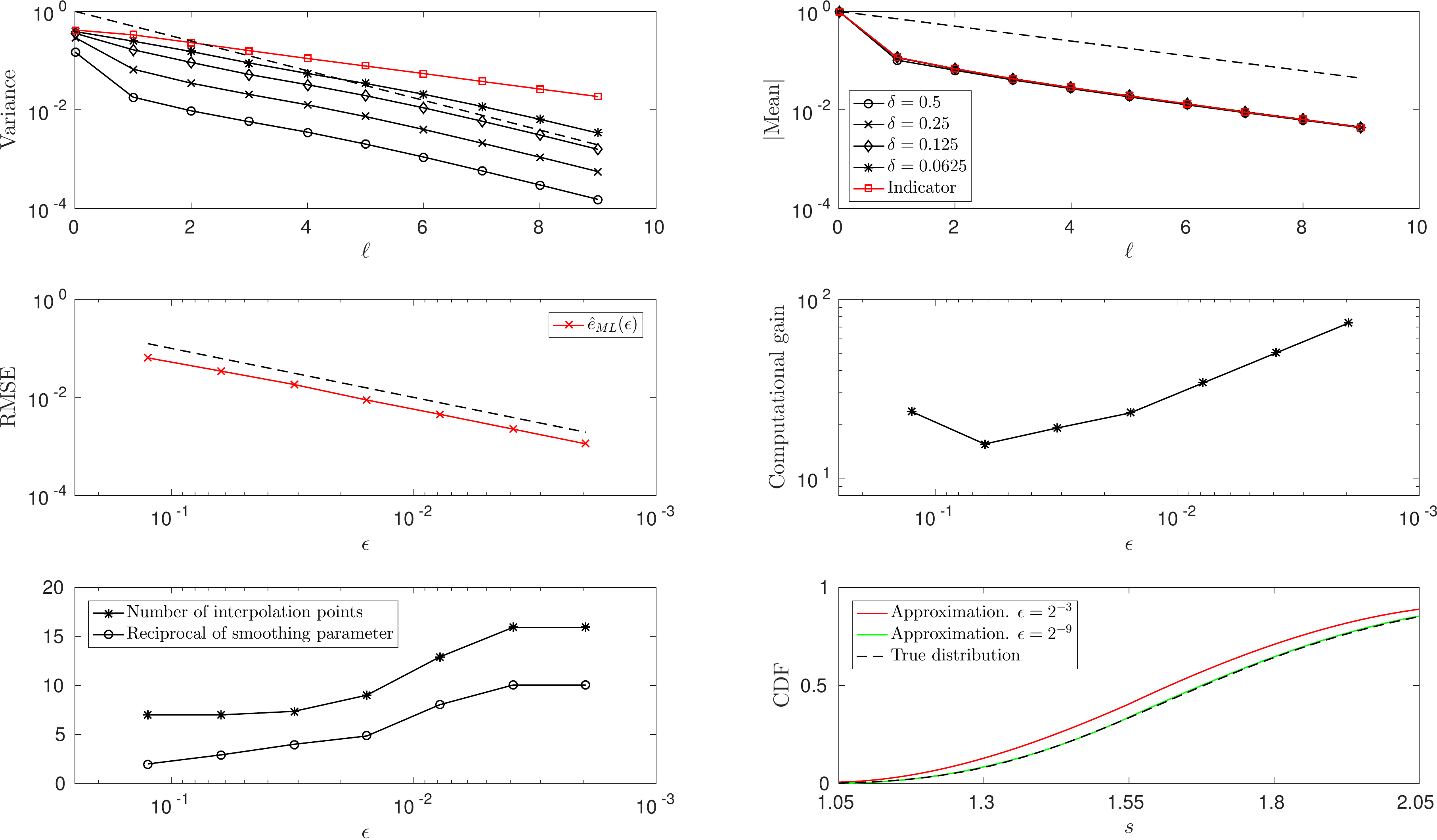}
\caption{Path-dependent functional:
variance decay (top left), mean decay (top right), computational gain (middle right), RMSE (middle left), average smoothing coefficient and 
number of interpolation points (bottom left), true CDF and estimated for $\varepsilon^{-3}$ and $\varepsilon^{-9}$ CDFs.}
\label{fig:barrier}
\end{figure}

\subsection{Exit times}\label{sec:exit}

Now we choose the parameters $\mu=0.01$, $\sigma=0.2$, and $T=2$
in the SDE \eqref{eq:gbm},
and we approximate the distribution function of
\[
Y = \inf \{ t \geq 0 : X_t \le b\} \wedge T))
\]
for $b = 0.95$ on the interval $[S_0,S_1] = [0.25,1.25]$.
The distribution of 
$\inf \{ t \geq 0 : X_t = b\}$ is an inverse Gaussian distribution
with parameters $\ln b/(\mu-\sigma^2/2)$ and $(\ln b)^2/\sigma^2$, see \cite{borodin2015handbook},
and this yields an explicit formula for the distribution function
of $Y$, since $T > S_1$. 

The numerical studies are performed and presented in the same way as in 
Sections \ref{sec:term} or \ref{sec:dep}, except in order to ensure better 
mean and variance decay, we use the distribution function of the minimum of 
the Brownian bridge between the discretization points. In the context of MLMC 
this has been done for the first time in \cite{giles08b} (see also \cite{primozic11}) and we refer to this work 
for coupling and MLMC implementation description.

The empirical values for the order of convergence of the variance
and the mean are both close to $1$ (see dashed reference lines for 
the first order decay) regardless of $\delta$, though the constants are inversely proportional to the value of $\delta$, thus showing the benefits of smoothing.
As before, we show the true distribution function $F$
on the interval $[S_0,S_1]$ along 
with two different approximations at different accuracies: 
$\Eps = 2^{-3}$ (red line) and $\Eps = 2^{-9}$ 
(green line).

\begin{figure}
        \centering
        \includegraphics[width=\textwidth]{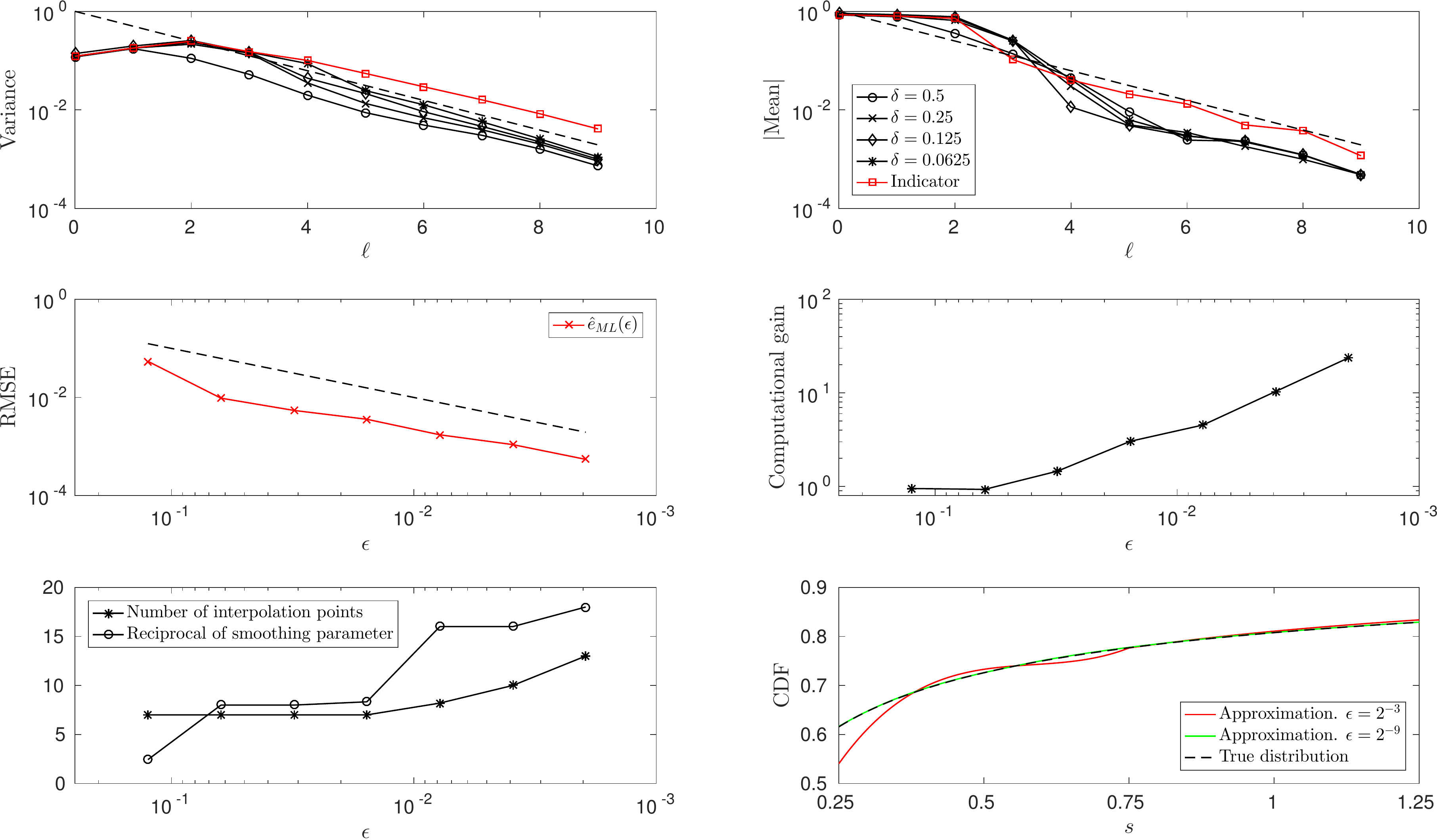}
\caption{Exit time:
variance decay (top left), mean decay (top right), computational gain (middle right), RMSE (middle left), average smoothing coefficient and number of interpolation points (middle left), true CDF and estimated for $\varepsilon^{-2}$ and $\varepsilon^{-6}$ CDFs.}
\label{fig:exit}
\end{figure}

\subsection{Conclusions}

The two most important findings are as follows.
The estimate  $\hat{e}_{\text{ML}}(\Eps)$ for the root mean
squared error of the adaptive MLMC algorithm
is in the range of the desired accuracy $\Eps$ for all
three functionals; 
actually, it is less than $\Eps$ in our experiments. 
For all three functionals the adaptive MLMC algorithm achieves
a substantial computational gain over the single-level algorithm.
For instance, if we ask for accuracy $\Eps=2^{-9}$, then this 
gain is around $840$, $70$ and $25$ times for smooth path-independent, smooth path-dependent and exit time functionals.

We clearly see that a proper choice of $\delta$ 
decreases the variances and consequently leads to 
a computational gain. For the smooth path-independent functional
this effect is very strong, and it is still substantial
for the smooth path-dependent functional. In both cases
smoothing ($\delta > 0$)
improves the empirical order of convergence of the variances,
compared to the indicator function ($\delta=0$).
For the exit time the effect is much weaker, and we have more or 
less the same empirical order of convergence of the variances.
For all three functionals the mean decay 
shows basically no dependence on $\delta$.

For all three functionals the empirical means of the number of
interpolation points and of the reciprocal of the smoothing
parameter, as chosen by our adaptive algorithm, increase similar 
to each other with respect to $\Eps$, which is consistent with the 
assumptions in Sections \ref{s44} and \ref{s55} and with the asymptotic 
analysis, see \eqref{f9a} and \eqref{f9b} in the Supplementary Materials.
The true distribution functions and their approximations with the 
smaller value of $\Eps$ are virtually indistinguishable.

\section*{Acknowledgments}
Mike Giles was partially supported as the Mercator Fellow of
the DFG Research Training Group 1932 ``Stochastic Models
for Innovations in the Engineering Sciences''
at the University of Kaiserslautern.

Tigran Nagapetyan was supported by EPSRC Grant
EP/N000188/1.

Klaus Ritter was partially supported by the ``Center
of Mathematical and Computational Modeling $(\rm{CM})^2$''
at the University of Kaiserslautern.

\nocite{giles08,bc15,bc16}

\bibliographystyle{siamplain}
\bibliography{mybibfile}
\section{Proof of the Theorem \ref{t2}}\label{SM:prt2}

We write $a \preceq b$ if there exists a constant
$c>0$ that does not depend on the parameters 
$k,\, \delta,\, L_0,\, L_1,\, N_{L_0},\, \dots, N_{L_1}$ 
such that $a \leq c \cdot b$. Moreover, 
$a \asymp b$ stands for $a \preceq b$ and $b \preceq a$.

Recall the error decomposition \eqref{err:dec:num}. 
Use Lemma \ref{l2-xx} together with \eqref{eq300}, \eqref{eq301},
and the boundedness of $g$ to obtain
\[
e_4
\preceq
\log k \cdot
\left( \frac{1}{N_{L_0}} + \sum_{\ell=L_0+1}^{L_1}
\frac{
\min(\delta^{-\beta_4} \cdot M^{-\ell \cdot \beta_5},1)}
{N_\ell} 
\right)
\]
for the variance $e_4$ of $\M$.
Furthermore, (A3) states that
\[
e_3 \preceq
\min \left( \delta^{-2 \alpha_1} \cdot M^{-L_1 \cdot 2 \alpha_2},
M^{-L_1 \cdot 2 \alpha_3} \right) 
\]
for the bias $e_3$ of $\M$. Use Lemma \ref{l1} and assumptions (E2)
and (E3) to obtain the error bound
\begin{align}\label{g99}
\err^2(Q_k^{r}(\M)) 
&\preceq 
k^{-2 (r+1)} + \delta^{2(r+1)}  
+
\min \left( \delta^{-2 \alpha_1} \cdot M^{-L_1 \cdot 2 \alpha_2},
M^{-L_1 \cdot 2 \alpha_3} \right) \notag \\
& 
\quad
+ 
\log k \cdot
\left( \frac{1}{N_{L_0}} + \sum_{\ell=L_0+1}^{L_1}
\frac{
\min(\delta^{-\beta_4} \cdot M^{-\ell \cdot \beta_5},1)}
{N_\ell} 
\right)
\end{align}
see \cite[Eqn. (2.15)]{giles2014multi}.

We determine parameters of the 
algorithm $Q_k^{r}(\M)$
such that an
error of about $\Eps \in \left]0,\min(1,\delta_0^{r+1})\right[$ 
is achieved at a small cost.
More precisely, we minimize the upper bound \eqref{eq:cost:mlmc:new} 
for the cost,
subject to the constraint that the upper bound \eqref{g99} for the 
squared error is at most $\Eps^2$, up to multiplicative constants for both
quantities. 

First of all we consider the case $\delta>0$, and we choose
\begin{equation}\label{f9a}
\delta = \Eps^{1/(r+1)}
\end{equation}
and, up to integer rounding,
\begin{equation}\label{f9b}
k = \Eps^{-1/(r+1)}
\end{equation}
and
\begin{equation}\label{f9d}
N_{L_0} = \Eps^{-2} \cdot \log_M \Eps^{-1}.
\end{equation}
This yields
\[
\err^2(Q_k^{r}(\M)) \preceq 
\Eps^2 + a^2(L_1)
+  
\log \Eps^{-1} \cdot \sum_{\ell=L_0+1}^{L_1}
\frac{
\min(\delta^{-\beta_4} \cdot M^{-\ell \cdot \beta_5},1)}
{N_\ell}
\]
with
\[
a(L_1) = 
\min \left( \delta^{-\alpha_1} \cdot M^{-L_1 \cdot \alpha_2},
M^{-L_1\cdot \alpha_3}
\right) .
\]
Since $k \preceq N_{L_0}$ and $k \cdot \delta \asymp 1$,
we obtain 
\begin{equation}\label{f7}
\cost(Q_k^{r}(\M)) \preceq 
c (L_0, L_1, N_{L_0}, \dots, N_{L_1})
\end{equation}
with
\[
c (L_0, L_1, N_{L_0}, \dots, N_{L_1}) =
\sum_{\ell=L_0}^{L_1} N_\ell\cdot  M^\ell
\]
from \eqref{eq:cost:mlmc:new}.
In contrast to \cite[Eqn. (2.16)]{giles2014multi}
this cost bound does not depend on $k$.

We need $a(L_1) \leq \Eps$, which requires $L_1 \geq q \cdot L^*$
with
\[
L^* = \frac{1}{r+1} \cdot \log_M \Eps^{-1}.
\]
Consequently, we choose
\begin{equation}\label{f9c} 
L_1 = q\cdot L^*,
\end{equation}
up to integer rounding.

For a single-level algorithm with smoothing, i.e., for $L_0=L_1$
and $\delta >0$,
all parameters have thus been determined,
and we obtain
$\err(Q_k^{r}(\M)) \preceq \Eps$
as well as
\begin{equation}\label{f8}
c (L_1, L_1, N_{L_1}) \asymp
\Eps^{-2} \cdot \log \Eps^{-1} \cdot M^{q \cdot L^*}
= \Eps^{-2 - q/(r+1)}\cdot\log \Eps^{-1}. 
\end{equation}
For a single-level algorithm without smoothing we obtain
the same result,
if we formally choose the parameters $\alpha_i$ by
\eqref{g203}, which leads to $q = (r+1)/\alpha$.

For a multi-level algorithm with $L_0 < L_1$
we obtain
\[
\err^2 (Q_k^{r}(\M)) \preceq
\Eps^{2} +
\log \Eps^{-1} \cdot
\sum_{\ell=L_0+1}^{L_1}
\frac{v_\ell}{N_\ell}
\]
with
\[
v_\ell=\min( M^{L^* \cdot \beta_4} \cdot M^{-\ell\cdot\beta_5},1)
\]
as well as
\[
c (L_0, L_1, N_{L_0}, \dots, N_{L_1}) \asymp
\Eps^{-2} \cdot \log \Eps^{-1} \cdot M^{L_0}
+ \sum_{\ell=L_0+1}^{L_1} N_\ell \cdot M^\ell.
\]
We fix $L_0$ for the moment, and we minimize
\[
h(L_0,N_{L_0+1}, \dots, N_{L_1}) =
\Eps^{-2} \cdot \log \Eps^{-1} \cdot M^{L_0} +
\sum_{\ell=L_0+1}^{L_1} N_\ell \cdot M^\ell
\]
subject to
\[
\sum_{\ell=L_0+1}^{L_1} \frac{v_\ell}{N_\ell } \leq
\Eps^2 / \log \Eps^{-1}.
\]
A Lagrange multiplier leads to
\[
N_\ell = \Eps^{-2} \cdot \log \Eps^{-1} \cdot G(L_0) 
\cdot \left(v_\ell \cdot M^{-\ell}\right)^{1/2},
\]
up to integer rounding, which satisfies the constraint with 
\[
G(L_0) = 
\sum_{\ell = L_0 + 1}^{L_1} 
\left(v_\ell \cdot M^\ell\right)^{1/2} =
\sum_{\ell = L_0 + 1}^{L_1} 
\left(
\min( M^{L^* \cdot \beta_4} \cdot M^{-\ell\cdot\beta_5},1)
 \cdot M^{\ell}\right)^{1/2}.
\]
Moreover, this choice of $N_{L_0+1}, \dots ,N_{L_1}$ yields 
\begin{equation}\label{f33}
h(L_0,N_{L_0+1}, \dots, N_{L_1}) 
=\Eps^{-2} \cdot \log \Eps^{-1} \cdot
\left( M^{L_0} + G^2(L_0) \right).
\end{equation}

Put $L^\dag = \beta_4/\beta_5 \cdot L^*$. 
Consider the case $q\le\beta_4/\beta_5$.
Then we have $L_1 \leq L^\dag$, and therefore
\[
M^{L_0} + G^2(L_0)
= M^{L_0}+\left(\sum_{\ell=L_0+1}^{L_1}
M^{\ell/2}\right)^2\asymp M^{L_0}+ M^{L_1}\asymp M^{L^* \cdot q}.
\]
Observing 
\eqref{f8} 
we get \eqref{z1} in the present case
already by single-level algorithms.

{}From now on we consider the case
$q>\beta_4/\beta_5$.
Suppose that $L_0 < L^\dag$. Then we get
\begin{align*}
M^{L_0} + G^2(L_0) &\asymp
M^{L_0} + 
\left( \sum_{\ell = L_0 + 1}^{L^\dag} 
M^{\ell/2} \right)^2 + 
M^{L^* \cdot \beta_4} \cdot 
\left(
\sum_{\ell = L^\dag + 1}^{L_1} 
M^{\ell \cdot (1-\beta_5)/2}
\right)^2 \\
&\asymp
M^{L^\dag} + G^2(L^\dag).
\end{align*}
It therefore suffices to study the case 
$
L_0 \geq L^\dag,
$
where we have
$$
M^{L_0} + G^2(L_0)
=
M^{L_0}+ M^{L^* \cdot \beta_4} \cdot \left(\sum_{\ell=L_0+1}^{L_1}
M^{\ell \cdot (1-\beta_5)/2}\right)^2.
$$
Note that
\begin{align*}
\beta_5=1 \quad&\Rightarrow\quad 
M^{L_0} + G^2(L_0) \asymp M^{L_0}+M^{L^* \cdot \beta_4} \cdot 
(L_1-L_0)^2,\\
\beta_5>1 \quad&\Rightarrow\quad 
M^{L_0} + G^2(L_0) \asymp M^{L_0}+M^{L^* \cdot \beta_4}\cdot 
M^{L_0 \cdot (1-\beta_5)}\asymp M^{L_0},\\
\beta_5<1 \quad&\Rightarrow\quad 
M^{L_0} + G^2(L_0) \asymp M^{L_0}+M^{L^* \cdot \beta_4} \cdot 
M^{L_1 \cdot (1-\beta_5)}.
\end{align*}
Hence we choose 
\[
L_0=\beta_4/\beta_5\cdot L^*
\]
in all these cases.
Hereby we obtain
\[
M^{L_0} + G^2(L_0) \asymp 
M^{L^* \cdot \beta_4/\beta_5} \cdot
\begin{cases}
\left(L^*\right)^2, 
&\text{if $\beta_5=1$,}\\
1,   &\text{if $\beta_5> 1$,}
\end{cases} 
\]
as well as
\[
M^{L_0} + G^2(L_0) \asymp 
M^{\max(\beta_4/\beta_5,\beta_4+(1-\beta_5)\cdot q) \cdot L^*}
\quad \text{if $\beta_5 < 1$.}
\]
In any case, under the condition $q>\beta_4/\beta_5$, 
these estimates are superior to $M^{L^* \cdot q}$, cf.\ 
\eqref{f8}.
Use \eqref{f33} and $M^{L^*} = \Eps^{-1/(r+1)}$ to derive
\eqref{z2}--\eqref{z5}.

\section{On the Type-2 Constant of $(\R^k,|\cdot|_\infty)$}
From \cite[p.~159]{heinrich98}
we know how to exploit the type of a Banach
space $E$ in the analysis of multilevel algorithms taking values in
$E$. The key ingredient is the existence of a constant $c>0$
(as small as possible) such that
\begin{equation}\label{eq201}
\V (\sum_{i=1}^n \M_i) \leq c^2 \cdot \sum_{i=1}^n \V (\M_i)
\end{equation}
holds for every $n \in \N$ and every independent sequence
$\M_1,\dots,\M_n$ of $E$-valued square-integrable random elements.
The smallest such constant is called the type-2 constant of
the space $E$.

In the sequel we focus on the space $E=\R^k$, equipped with the
$\ell_\infty$-norm $|\cdot|_\infty$, and we provide an explicit
value for 
\[
c= c(k),
\]
i.e., an explicit upper bound for the type-2 constant of
this space.

The following result is due to \cite[Lemma 6]{Linde1974},
and stated there in a slightly different setting.
For convenience of the reader we present the proof from \cite{Linde1974} in the setting of the present paper. 
Let $Z_1,\dots$ be an independent sequence of standard normally 
distributed random variables. Moreover, let
\begin{equation}\label{eq204}
\gamma^2(k) = \ln(k+1) + 
(8/\pi)^{1/2} \sum_{j=2}^{k+1} \frac{1}{(\ln(j))^{1/2} \cdot j^2}.
\end{equation}

\begin{lemma}\label{l10}
For $n \in \N$ and  $x_1, \dots, x_n \in \R^k$ let
\[
X = \sum_{i=1}^n Z_i \cdot x_i.
\]
Then
\[
\Exp(|X|_\infty^2) \leq \gamma^2(k) \cdot \sum_{i=1}^n |x_i|_\infty^2.
\]
\end{lemma}

\begin{proof}
It suffices to show
$\Exp(|X|_\infty^2) \leq \gamma^2(k)$
in the case $\sum_{i=1}^n |x_i|_\infty^2 = 1$.
Put
\[
F(s) = P(\{ |X|^2_\infty \leq s\})
\]
and
\[
G(s) = P(\{ Z_1^2 \leq s\})
\]
for $s \geq 0$. 

Use integration by parts to obtain
\begin{align*}
1-G(s) &= (2\pi)^{-1/2} \int_s^\infty y^{-1/2} \cdot \exp(-y/2) \, dy\\
& = (2/\pi)^{1/2} \cdot s^{-1/2} \cdot \exp (-s/2)
- (2\pi)^{-1/2} \int_s^\infty y^{-3/2} \cdot \exp(-y/2) \, dy\\
&\leq
(2/\pi)^{1/2} \cdot s^{-1/2} \cdot \exp (-s/2)
= 2 G^\prime(s),
\end{align*}
see \cite[Lemma 4]{Linde1974},
where $1-G(s) \geq G^\prime(s)$ for $s \geq 2$ is shown as well.

The $j$-th component $X_j$ of $X$ is normally 
distributed with zero mean and variance 
\[
\sigma_j^2 = \sum_{i=1}^n x_{i,j}^2 \leq 1.
\]
Let $G(\infty)=1$, $G^\prime(\infty)=0$,
and $s/0=\infty$ if $s >0$ for notational convenience.
For $s>0$, 
\[
P(\{X_j^2 > s\}) =
1-G( s/\sigma_j^2) 
\leq
2 G^\prime( s/\sigma_j^2),
\]
and therefore
\begin{align*}
1 - F(s) &\leq \sum_{j=1}^k 
P(\{X_j^2 > s\}) \leq 
2  \sum_{j=1}^k 
G^\prime( s/\sigma_j^2).
\end{align*}
Hence, for every $\nu > 0$,
\begin{align*}
\Exp (|X|^2_\infty) &\leq \nu + 
\int_\nu^\infty (1-F(s))\, ds
\leq \nu + 2 \sum_{j=1}^k \int_\nu^\infty 
G^\prime( s/\sigma_j^2)\, ds\\
& =
\nu + 2 \sum_{j=1}^k \sigma_j^2 (1-  G( \nu/\sigma_j^2) )
\leq
\nu + 4 \sum_{j=1}^k \sigma_j^2 \, G^\prime( \nu/\sigma_j^2)\\
&=
\nu + (8/\pi)^{1/2} \sum_{j=1}^k 
\left( \sigma_j /\nu^{1/2} \cdot \exp(-\nu/(2\sigma_j^2)\right).
\end{align*}
Choose
\[
\nu = \sup_{1 \leq j \leq k} \bigl( \sigma_j^2 \cdot \ln (j+1)
\bigr) \leq \ln (k+1) 
\]
to obtain $\nu/\sigma_j^2 \geq \ln (j+1)$ and therefore
\[
\Exp (|X|^2_\infty) \leq \ln(k+1) + 
(8/\pi)^{1/2} \sum_{j=2}^{k+1} \frac{1}{(\ln(j))^{1/2} \cdot
j^2},
\]
as claimed.
\end{proof}

In addition to the normally distributed random variables $Z_i$
and the corresponding random vectors $X$
we also consider an independent sequence $\eps_1, \dots$ of
Bernoulli (or Rademacher) random variables,
i.e., $\eps_i$ takes the values $\pm 1$ with probability $1/2$.
For $n \in \N$ and $x_1,\dots,x_n \in \R^k$ we define
\[
\tilde{X}= \sum_{i=1}^n \eps_i \cdot x_i.
\]
We have
\begin{equation}\label{eq200}
\Exp(|\tilde{X}|_\infty^2 ) \leq \pi/2 \cdot
\Exp(|X|_\infty^2)
\end{equation}
as a particular case of Pisier (1973, Prop.\ 1), which
deals with arbitrary normed spaces and symmetric random variables
$Z_i$.

The Rademacher type 2 constant of the space $\R^k$, which we denote
by $T_2(k)$, is the smallest constant $\tilde{c}>0$ such that
\[
\Exp(|\tilde{X}|_\infty^2) \leq 
\tilde{c}^2 \cdot \sum_{i=1}^n |x_i|_\infty^2
\]
for every $n \in \N$ and all $x_1, \dots, x_n \in \R^k$.
Observe that the latter is an estimate for the variance of
$\tilde{X}$ in terms of the variances $|x_i|_\infty^2$ of the
random vectors $\eps_i \cdot x_i$.
This estimate actually extends to every independent sequence 
of random vectors, at the expense of a slightly
larger constant. More precisely, for every $n \in \N$ and
every independent sequence $\M_1,\dots,\M_n$ of
square-integrable random vectors 
with values in $\R^k$, we have \eqref{eq201} with
\[
c(k) = 2\, T_2(k).
\]
This result actually holds true with the corresponding type 2
constant for every Banach space $E$ of this type ,
see \cite[Prop.~9.11]{ledoux1991}.

Lemma \ref{l10} together with \eqref{eq200} provides an explicit
constant for \eqref{eq201} to hold in the case $E=\R^k$, namely
\begin{equation}\label{eq202}
c(k) = (2\pi)^{1/2} \cdot \gamma(k).
\end{equation}
This constant is of the order $(\ln(k))^{1/2}$, which is
known to be optimal for the spaces~$\R^k$.
\end{document}